\documentclass[reqno]{amsart}
\usepackage{amssymb,amsmath,amsthm,amsfonts,amscd}
\usepackage{graphicx,color}
\usepackage{physics}
\usepackage[dvipsnames]{xcolor}% tem mais cores que o acima.
\usepackage{comment}
\usepackage[utf8]{inputenc}
\usepackage[T1]{fontenc}
\usepackage{bigints}
\usepackage{esint}
\usepackage{tikz}
\usepackage{enumerate}
\usepackage{mathtools}

\newtheorem{theorem}{Theorem}[section]

\newtheorem{hipotese}{Hypothesis}[section]
\newtheorem{lemma}{Lemma}[section]
\newtheorem{remark}{Remark}

\newcommand{\hide}[1]{}

\usepackage{mathtools}
\mathtoolsset{showonlyrefs}

\title[Periodic Homogenization of Local/Nonlocal Systems]{\bf Periodic Homogenization of Local/Nonlocal Systems}

\author[Marcone C. Pereira, Luiza C. Rosa da Silva and Julio D. Rossi]{Marcone C. Pereira, Luiza C. Rosa da Silva and Julio D. Rossi}

\begin{document}
%----------------------------------
\address{Marcone C. Pereira 
\hfill\break\indent Depto. de Matem\'atica, Instituto de Matem\'atica e Estat\'istica 
\hfill\break\indent Universidade de S\~ao Paulo, Rua do Mat\~ao 1010, S\~ao Paulo - SP - Brazil}
\email{marcone@ime.usp.br}

\address{Luiza C. Rosa da Silva 
\hfill\break\indent Depto. de Matem\'atica, Instituto de Matem\'atica e Estat\'istica 
\hfill\break\indent Universidade de S\~ao Paulo, Rua do Mat\~ao 1010, S\~ao Paulo - SP - Brazil}
\email{luiza@ime.usp.br}

\address{Julio D. Rossi
\hfill\break\indent Depto. de Matematicas, Universidad Torcuatto di Tella
\hfill\break\indent Av. Pres. Figueroa Alcorta 7350, C1428, Buenos Aires - Argentina}
\email{julio.rossi@utdt.edu}

%-----------------------------------
\begin{abstract}

In this paper, we study the homogenization of elliptic equations that combine a local part, given by the Laplacian with Neumann boundary conditions, and its  nonlocal version, defined through an integral operator with a smooth kernel. These two components are coupled through an additional nonlocal operator also given by a smooth kernel. We consider a sequence of partitions of a fixed spatial domain into two regions - local and nonlocal - which are periodically distributed in space (with one of the regions consisting of small, periodically arranged holes).
Depending on the relative location of the local and nonlocal regions, we obtain qualitatively different limit behaviors. When the local part of the equation is confined to the small periodic holes, the sequence of solutions converges to the unique solution of a limit system in which the local component vanishes, while the nonlocal part persists and splits into two distinct components. On the other hand, when the local part of the problem lies outside the holes, the limit system exhibits a homogenized local diffusion operator coupled with a nonlocal equation. Finally, we analyze an intermediate regime in which only part of the local diffusion survives in the limit, by considering configurations consisting of parallel thin strips instead of holes. 

\end{abstract}


\maketitle
%---------------

\section{Introduction}\label{sec.intro}
The classical example of a local linear elliptic operator is
\begin{equation} \label{heat.equation}
\Delta u (x) = \operatorname{div} (\nabla u(x)).
\end{equation}
When considered in a bounded domain $\Omega$, this operator must be complemented with boundary conditions to ensure well-posedness of the problem.
Throughout this work, we impose homogeneous Neumann boundary conditions,
\[
\frac{\partial u}{\partial \eta}(x) = 0,
\]
where $\eta$ denotes the unit outward normal vector to the boundary of $\Omega$. These conditions physically correspond to no-flux across the boundary and are natural in various diffusion problems where the total mass is preserved.

In contrast, for nonlocal diffusion models, one typically considers operators of the form
\begin{equation} \label{eq:nonlocal.equation}
\int_{\Omega} G(x-y) \big( v(y) - v(x) \big)dy,
\end{equation}
where the kernel $G$ is nonnegative, continuous, symmetric, and satisfies
\[
\int_{\mathbb{R}^N} G(z)\,dz = 1.
\]
The quantity in \eqref{eq:nonlocal.equation} is genuinely nonlocal, since its evaluation at a point $x$ depends on the values of $v$ over the support of $G(x-\cdot)$. Operators of this type and their variations have been extensively used to model diffusion processes that involve long-range interactions, anomalous transport, or jump processes; see for instance \cite{ BCh, CF, Chasseigne-Chaves-Rossi-2006, Cortazar-Elgueta-Rossi-Wolanski, wolanski, delia-du-gunzburger-lehoucq, toninho} and the book \cite{julio}. For nonlocal equations with Neumann-type conditions and singular kernels, we refer to \cite{SXE} and the references therein. Notice that the nonlocal formulation provides a natural framework for problems that involve solutions  whose classical derivatives may not be defined or where interactions occur over finite distances. In this way, 
these models capture essential features of systems governed by nonlocal interactions.

In this work, our main goal is to study the homogenization process associated with problems that couple a classical local diffusion equation on one region of the domain with a nonlocal diffusion operator on a complementary region, interacting through a nonlocal transmission term. More precisely, we consider a smooth bounded domain $\Omega\subset\mathbb{R}^N$ decomposed into two disjoint sets $A$ and $B$ ($\Omega = A \cup B$, $A\cap B = \emptyset$)
with $A$ and $B$ having nonempty interior, and analyze the following stationary local/nonlocal system with transmission term: 
\begin{equation}\label{local}
\begin{cases}
\Delta u(x) + \displaystyle\int_{B} J(x-y)\big(v(y)-u(x)\big)dy = f(x), & x\in A, \\
\displaystyle \frac{\partial u}{\partial \eta}(x)=0, & x\in \partial A,
\end{cases}
\end{equation}
and
\begin{equation}\label{nonlocal}
\int_{B} G(x-y)\big(v(y)-v(x)\big) dy
+\int_{A} J(x-y)\big(u(y)-v(x)\big) dy
= f(x), \qquad x\in B.
\end{equation}
Here, the kernels $J$ and $G$ govern the jumps within and across the domains. In particular, $J$ encodes the coupling between $A$ and $B$, while $G$ acts only inside $B$. This coupled framework allows us to model systems in which local and nonlocal diffusion coexist and interact across an interface, capturing both short-range and long-range diffusion effects. 

Throughout the paper, we assume the following hypothesis on the kernels: 
\begin{hipotese}\label{hipo} {\rm
The kernels $K=J$, $G$ satisfy: $K \in C(\mathbb{R}^N,\mathbb{R})$, nonnegative, radial and compactly supported with $K(0)>0$, and
\[
\int_{\mathbb{R}^N} K(z)\,dz = 1.
\]
These conditions implies that $K$ is a radial probability density.
Moreover, the kernel $J$ satisfies $B_R(0) \subset \operatorname{supp} J$ with $R > \operatorname{diam} \left( A \cup B \right)$.\footnote{Throughout this work, $\operatorname{supp} K$ denotes the support of the function $K$, $\operatorname{diam} (S)$ the diameter of the set $S \subset \mathbb{R}^N$, and $B_R(0)$ the ball of radius $R$ centered at the origin.} }
\end{hipotese}

An important remark concerns the assumption on the kernel $J$, $B_R(0) \subset \operatorname{supp} J$ with $R > \operatorname{diam} \left( A \cup B \right)$, which is imposed to guarantee interaction between all points of $A$ and $B$. This ensures that the system remains fully coupled throughout the domain and prevents a degeneracy that could arise if there are points of the sets 
$A$ and $B$, that are located too far from each other.
We recall that similar coupled local/nonlocal diffusion frameworks have been previously introduced in \cite{bruna, luiza, quiros}. For more general systems, we refer \cite{delia-du-gunzburger-lehoucq,delia-perego-bochev-littlewood,delia-ridzal-peterson-bochev-shashkov,du-li-lu-tian,gal-warma,kriventsov}.

As usual for Neumann diffusion problems, a necessary condition for existence  of a solution is that $f \in L^2(\Omega)$ verifies 
\[
\int_A f(x)\,dx + \int_B f(x)\,dx = 0.
\]

Our main goal here is to study the homogenization of this coupled local/nonlocal model. We focus on two classical periodic configurations widely explored in homogenization theory: periodically distributed perforations and thin periodic strips. These geometries are well established in the literature as periodic configurations with concrete applications, see, for example, \cite{doina, luctartar, taras, marc2, tarasbook}.
Homogenization procedures for purely nonlocal equations involving a single kernel have been extensively investigated; see \cite{marc1, marc2, marc3}. For nonlocal equations with singular kernels, we refer to \cite{schwab, waurick, Ca}. In addition, for differential problems with Signorini-type interface constraints, see \cite{perugia}.
Our contribution extends these frameworks to the setting where local and nonlocal operators act simultaneously on periodically perforated domains.

Throughout the paper, the homogenization framework is based on a sequence of partitions $(A_n, B_n)$ of a fixed domain $\Omega$,
$$\Omega = A_n \cup B_n,$$
representing either periodically distributed perforations or thin periodic strips.
The corresponding characteristic functions $\chi_{A_n}(x)$ and $\chi_{B_n}(x)$ satisfy
\begin{equation}
\begin{gathered}
\chi_{A_n}(x) \rightharpoonup X, \quad \text{weakly-* in } L^{\infty}(\Omega), \\
\chi_{B_n}(x) \rightharpoonup 1 - X, \quad \text{weakly-* in } L^{\infty}(\Omega),
\end{gathered}
\end{equation}
as $n \to +\infty$, where $X$ is a positive constant such that
\begin{equation}\label{hipoX}
0 < X < 1.
\end{equation}

These hypotheses are consistent with previous analyses of nonlocal homogenization problems; see \cite{monia} and references therein. They ensure that the solutions of the local and nonlocal problems remain uniformly bounded, preventing singular behaviors, a key aspect for the construction of correctors and the derivation of uniform estimates.

Our main goal is to pass to the limit as $n \to \infty$ in the sequence of solutions $(u_n, v_n)$ to the system \eqref{local}--\eqref{nonlocal} associated with the partitions $(A_n, B_n)$, and to identify how the interplay between local and nonlocal effects influences the resulting homogenized model.

Our analysis reveals that different homogenized systems emerge depending on the location of the local and nonlocal operators.  We will be able to consider two different types of periodic configurations, periodic perforations and equidistributed periodic strips, which will be
precisely defined in Section 2.

In the periodic perforated case
when the local operator acts inside the perforations, the contribution of the Laplacian vanishes in the limit, yielding a purely nonlocal effective model, with weights determined by the constant $X$.

On the other hand, when the nonlocal operator acts inside the perforations, both local and nonlocal effects persist in the limit, producing a coupled homogenized system that combines a classical local operator with a nonlocal term modulated by $X$.
In particular, in this last case in the case of a perforated domain with Neumann boundary conditions on the boundaries of the holes, we recover the classical homogenized equation described in \cite{doina}. 

For the strip configuration, the oscillatory geometry gives rise to an effective operator that retains gradients only in directions tangential to the strips, while the nonlocal interaction averages across them.

From an analytical point of view, the proof relies on a combination of variational arguments and uniform energy estimates, see \cite{brezis, evans}, together with the construction of correctors adapted to the geometry of the holes, as discussed in \cite{monia} for nonlocal terms and in \cite{doina} for local parts.
The convergence of the characteristic functions of the phases and the stability properties of the nonlocal operator play a central role. The main difficulty in the paper consists in choosing appropriate test functions that allows to identify the different limit problems. Altogether, these tools enable us to rigorously justify the limiting behavior in both the periodic holes and strips settings, and to characterize the effective coupling mechanisms in each case.

\section{Description of mixed local and nonlocal equations and statements of the main results}
\setcounter{equation}{0}

As we have mentioned, our main objective is to analyze the homogenization process associated with the coupled local--nonlocal system 
\eqref{local}--\eqref{nonlocal}. 
The model involves the Laplace operator acting on the local subdomain and a nonlocal operator of convolution type defined by symmetric, 
non-singular kernels in the complementary region. 
For each $n \in \mathbb{N}$, the smooth bounded domain $\Omega \subset \mathbb{R}^N$ ($N \ge 2$) is decomposed into two disjoint subdomains $A_n$ and $B_n$,
\[
\Omega = A_n \cup B_n, 
\qquad 
A_n \cap B_n = \emptyset,
\]
whose geometry oscillates periodically as $n \to \infty$. 
Regardless of the configuration (periodically distributed holes or thin strips), the subdomains $A_n$ and $B_n$ represent, respectively, the local and nonlocal parts of the system.

Associated with this decomposition, we consider the homogeneous Neumann problem with the transmission nonlocal term
\begin{equation}\label{hlocal}
\begin{cases}
    f_n(x)=\Delta u_n(x) + \displaystyle\int_{B_n} J(x-y)(v_n(y)-u_n(x))\,dy, \quad x\in A_n, \\[0.5em]
    \displaystyle\frac{\partial u_n}{\partial \eta}(x)=0, \quad x\in \partial A_n,
\end{cases}
\end{equation}
and
\begin{equation}\label{hnl}
    f_n(x)=\int_{B_n} G(x-y)(v_n(y)-v_n(x))\,dy+\int_{A_n} J(x-y)(u_n(y)-v_n(x))\,dy, \quad x\in B_n,
\end{equation}

The existence of solutions in \eqref{hlocal}--\eqref{hnl} requires that the source term 
$f_n \in L^2(\Omega)$ satisfies the compatibility condition
\[
\int_{A_n} f_n(x)\,dx+\int_{B_n} f_n(x)\,dx = 0.
\]
For such a source $f_n$, the weak formulation of \eqref{hlocal}--\eqref{hnl} is naturally written in the Hilbert space
\[
W_n := 
\left\{
(u,v)\in H^1(A_n)\times L^2(B_n)
\;\middle|\;
\int_{A_n} u(x)\,dx+\int_{B_n} v(x)\,dx=0
\right\},
\]
and says that $(u,v)\in W_n$ satisfies
\[
a_n((u,v),(\varphi,\phi))
= -\int_{A_n} f_n(x)\varphi(x)\,dx 
- \int_{B_n} f_n(x)\phi(x)\,dx
\qquad \forall\, (\varphi,\phi)\in W_n,
\]
with the bilinear form $a_n$ defined by
\begin{align*}
a_n((u,v),(\varphi,\phi)) 
&= \int_{A_n} \nabla u(x)\cdot\nabla \varphi(x)\,dx
+ \frac{1}{2}\!\iint_{B_n\times B_n}\! G(x-y)\big(v(y)-v(x)\big)\big(\phi(y)-\phi(x)\big)\,dx\,dy\\
&\quad + \iint_{A_n\times B_n}\! J(x-y)\big(u(y)-v(x)\big)\big(\varphi(y)-\phi(x)\big)\,dx\,dy.
\end{align*}

This form defines an inner product in $W_n$ and the nonlocal transmission term allows us to get
the existence and uniqueness of a weak solution from the Lax--Milgram theorem. The solution can equivalently be characterized as 
the unique minimizer in $W_n$ of the energy functional
\begin{align}\label{energy.functional}
E_n(u,v)
&= \frac{1}{2}\int_{A_n} |\nabla u(x)|^2\,dx
+ \frac{1}{4}\iint_{B_n\times B_n} G(x-y)\big(v(y)-v(x)\big)^2\,dx\,dy \notag\\
&\quad + \frac{1}{2}\iint_{A_n\times B_n} J(x-y)\big(v(y)-u(x)\big)^2\,dx\,dy
+ \int_{A_n} f_n(x)u(x)\,dx + \int_{B_n} f_n(x)v(x)\,dx.
\end{align}

Once the well-posedness of \eqref{hlocal}--\eqref{hnl} is established, we look for the asymptotic behavior of the family of 
solutions $\{(u_n,v_n)\}$ as $n\to\infty$. 
The aim is to identify the homogenized limit problem satisfied by the limit pair $(u,v)$ and to characterize the effective coefficients 
arising from the local--nonlocal interaction through the microstructure. For the homogenization procedure, we assume that there is $f\in L^2(\Omega)$ such that 
$$f_n \longrightarrow f, \qquad \text{strongly in} \  L^2(\Omega), \quad \text{as} \ n\to \infty.$$

Once the equation is fixed, we study $\Omega$ under three periodic configurations  corresponding to the classical settings in homogenization theory:

\begin{enumerate}[(i)]
    \item \emph{Periodically distributed holes}, where the perforations are contained in \(A_n\) and the nonlocal operator acts in the complementary region \(B_n = \Omega \setminus A_n\);
    
    \medskip
    
    \item \emph{The complementary configuration}, in which the holes belong to \(B_n\) and the Laplacian acts in \(A_n = \Omega \setminus B_n\);
    
    \medskip
    
    \item \emph{Thin periodic strips}, where \(A_n\) consists of a family of strips distributed periodically within \(\Omega\).
\end{enumerate}
 
In what follows, we provide the precise definitions of these periodic structures: the cases of perforated domains and thin strips.

\subsection*{The Holes}

We now formalize the geometric setting for the perforated (hole-type) configuration. Consider a fixed and smooth bounded domain $\Omega\subset \mathbb{R}^N$, and let
$Y = (0, \ell_1) \times \cdots \times (0, \ell_N) \subset \mathbb{R}^N$ be a reference periodicity cell. Let \( T \subset Y \) be an open subset with smooth boundary \( \partial T \) such that \( \overline{T} \subset Y \), and define the usual cell described by the set: 
\begin{equation}\label{cell}
Y^* = Y \setminus T.
\end{equation}
Let \( n \in \mathbb{N} \), and consider the set of all translated images of $\frac{1}{n}\overline{T}$ given by
\[
\tau\left(\frac{1}{n} \overline{T}\right) := \bigcup_{k \in \mathbb{Z}^N} \frac{1}{n}(k \ell + \overline{T}), \quad \text{ where } \quad k\ell = (k_1 \ell_1, \dots, k_N \ell_N).
\]
%\textcolor{red}{ (melhor isso:) with\[ \tau\left(\frac{1}{n} \overline{T}\right) \subset \Omega\]}

The union above represents a periodic distribution of holes of equal size, regularly repeated throughout the space. 
Here, we will work with 
 holes that do not intersect the boundary set $\partial \Omega$.
Then, we consider the set of holes which are contained in the domain $\Omega$ as
\begin{equation}\label{holes}
T_n := \bigcup \left\{ \frac{1}{n}(k\ell + \overline{T}) \,\middle|\, \frac{1}{n}(k\ell + \overline{T}) \subset \Omega \right\},
\end{equation}
and the corresponding perforated domain by
\begin{equation}\label{domain_n}
\Omega_n := \Omega \setminus \overline{T_n}.
\end{equation}

Observe that, as $n$ increases, the domain $\Omega_n$ is progressively filled with a periodic constellation of small holes, becoming increasingly perforated. We denote by $T_n(j)$ the $j$th hole generated at the $n$th step. 

Depending on the case under consideration, we take $\Omega_n=A_n$ (or $B_n$). 

As a representative example, we consider the unit domain $\Omega = [0,1]^N$ for $N\ge 2$ and a periodic family of small balls distributed inside $\Omega$ according to the lattice $2n^{-1}\mathbb{Z}^N$. For each $n\in \mathbb{N}$, we work in the setting 
\[
\bigcup_{\substack{B_{r_n}(x_i) \subset [0,1]^N}} B_{r_n}(x_i),
\]
where $B_{r_n}(x_i)$ denotes the ball of radius $r_n$ centered at $x_i$, which belongs to the lattice, 
and the radius satisfies 
\[
0 < r_n = \frac{C}{n}, \qquad \mbox{with } C < 1.
\]

In two dimensions ($\Omega = [0,1]^2$), this corresponds to a periodic array of small disks. For instance, when $n=16$, the configuration is shown below.
\begin{center}
\begin{tikzpicture}[scale=5]
  % Unit square
  \draw[thick] (0,0) rectangle (1,1);

  % Parameters
  \def\r{0.015}   % radius of the balls
  \def\step{0.125} % 1/8: grid step for 8x8 = 64 balls

  % 8x8 grid of balls
  \foreach \i in {0,1,2,3,4,5,6,7} {
    \foreach \j in {0,1,2,3,4,5,6,7} {
      \fill[black!60] ({(\i+0.5)*\step}, {(\j+0.5)*\step}) circle (\r);
    }
  }

  % Caption
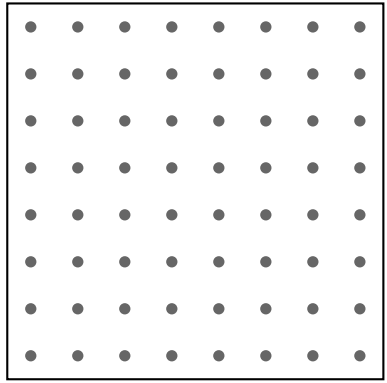
  \node at (0.5,-0.08) {$64$ disjoint balls periodically distributed in $[0,1]^2$};
\end{tikzpicture}
\end{center}

\subsection*{Thin Strips}
We now consider the periodic geometry consisting of thin strips. We define a partition of $\Omega$, where we describe $B_n = \Omega \setminus A_n$ and $A_n$ as an union of selected horizontal strips where, for $n \geq 2$, we get a partition of $\Omega$ into horizontal strips of equal height $h = 1/2^{n-1}$. Labeling the strips from bottom ($k = 1$) to top ($k = 2^{n-1}$), we define the set:
\begin{equation}\label{strip}
A_n = \bigcup_{\substack{k=1 \\ k \text{ even}}}^{2^{n-1}} (S_n)_k \quad \text{where} \quad (S_n)_k = \left\{x:=(\hat{x}, x_N) \in [0,1]^{N-1}\times [0,1] \;\middle|\; \frac{k-1}{2^{n-1}} \leq x_N < \frac{k}{2^{n-1}}\right\}.
\end{equation}
As a representative example in two dimensions, for $\Omega=[0,1]^2$, the domain is divided into $2^{n-1}$ horizontal strips of equal height.  
For instance, when $n=3$, the set $A_3 = (S_3)_2\cup (S_3)_4$ is represented below.
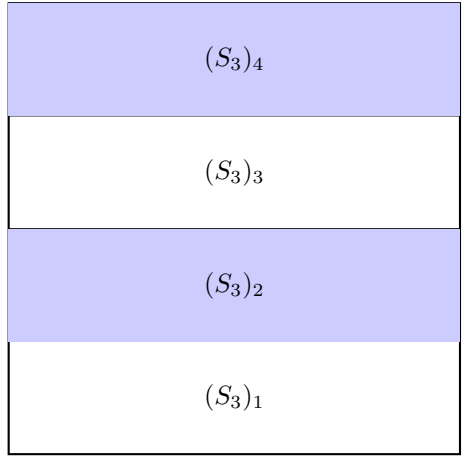
\begin{center}
\begin{tikzpicture}[scale=1.5]
    % Draw Omega square
    \draw[thick] (0,0) rectangle (4,4);
    % Partition lines
    \foreach \y in {1,2,3} { \draw[thin] (0,\y) -- (4,\y); }
    % Selected strips
    \fill[blue!20] (0,1) rectangle (4,2);
    \fill[blue!20] (0,3) rectangle (4,4);
    % Labels
    \node at (2,0.5) {$(S_3)_1$};
    \node at (2,1.5) {$(S_3)_2$};
    \node at (2,2.5) {$(S_3)_3$};
    \node at (2,3.5) {$(S_3)_4$};
    % Caption
    \node[below] at (2,-0.5) 
      {$A_3$: the selected strips $(S_3)_2$ and $(S_3)_4$ in $[0,1]^2$.};
\end{tikzpicture}
\end{center}

We now describe the homogenization results and discuss the limit as $n\to \infty$ in three separate cases, as described below.

\subsection{Local in Holes}
The first situation arises when the equation \eqref{hlocal} is defined in $A_n$, which is given by the union of holes \eqref{holes}, obtaining the following theorem.

\begin{theorem}\label{localinballs}
Let $(u_n,v_n)\in W_n$ be the family of solutions to 
\eqref{hlocal}--\eqref{hnl}, where $A_n$ is the union of holes \eqref{holes}.  
Then, as $n\to\infty$, the following convergences hold:
\[
\chi_{A_n}u_n \rightharpoonup u  \quad\text{in } L^2(\Omega)\quad \text{and} \quad 
\chi_{B_n}v_n \rightharpoonup v \quad\text{in } L^2(\Omega),
\]
and the limit pair $(u,v)\in L^2(\Omega)\times L^2(\Omega)$ satisfies the compatibility condition
\[
\int_\Omega u(x)\,dx+\int_\Omega v(x)\,dx=0.
\]
Moreover, the limit $(u,v)\in L^2(\Omega)\times L^2(\Omega)$
is uniquely characterized as the weak solution of the following {homogenized problem}:  
\begin{align}\label{limhlocal}
f(x)X=
\int_\Omega J(x-y)\,\big( Xv(y)-(1-X)u(x)\big)\,dy,
\qquad x\in\Omega,
\end{align}
\begin{align}\label{limhnl}
         f(x)(1-X)&=\int_\Omega G(x-y)((1-X)v(y)-(1-X)v(x))dy\notag \\
        &\qquad +\int_\Omega J(x-y)((1-X)u(y)- Xv(x))dy, \quad x\in \Omega.
\end{align}
This system gives the homogenized equation of \eqref{hlocal}--\eqref{hnl}.
\end{theorem}

Notice that the Laplacian term disappears in \eqref{limhlocal}. For this homogenization procedure, we use characteristic functions to pass to the limit in nonlocal terms, as shown in \cite[Theorem 2.3]{monia}. 

In addition to the convergence results obtained in Theorem~\ref{localinballs},
we introduce a corrector for the pair of limit solutions
$(u,v)\in L^2(\Omega)\times L^2(\Omega)$.
Due to the oscillatory nature of the perforated configuration, the convergence of the  sequences $(u_n,v_n)\in H^1(A_n)\times L^2(B_n)$ is, in general, only weak in the $L^2$-norm. 
To overcome this limitation, we construct a corrector sequence to compensate
these oscillations,
more precisely, following the classical corrector techniques in homogenization
(see, for example, \cite{monia}), we refine the convergence to the homogenized solution $(u,v)$ by introducing a suitable correction $w_{1,n}$ and $w_{2,n}$, which allows us to recover strong convergence in $L^2(\Omega)$. Moreover, we also need the regularity of $u$ to obtain this convergence.

Our correction must belong to $W_n$. Assuming~\eqref{hipoX}, we define the corrector vector
\[
w_n \; =\Bigg(
w_{1,n} - \frac{1}{2}\,\frac{m_n}{|A_n|},
\;\;
w_{2,n} - \frac{1}{2}\,\frac{m_n}{|B_n|}
\Bigg) 
\]
where 
\[
m_n := \int_{A_n} w_{1,n}(x)\,dx + \int_{B_n} w_{2,n}(x)\,dx
\]
with $w_{2,n}\in L^2(\Omega)$ defined by
$$w_{2,n}=\displaystyle\frac{\chi_{B_n}v}{1-X}$$ 
and 
$w_{1,n}\in L^2(\Omega)$ by
\begin{equation}\label{w1n}
    w_{1,n}(x)=
    \begin{cases}
        \dfrac{u(x)}{X}, & x\in B_n,\\[1.2ex]
      \displaystyle \sum_j \chi_{T_n(j)}(x)\,\fint_{T_n(j)} \dfrac{u(z)}{X}\,dz, & x\in A_n=\displaystyle\bigcup _{j=1}^{N_n}T_n(j).
    \end{cases}\end{equation}
Here, $\displaystyle\fint_{\mathcal{O}} f(x)\,dx$ is the average value of $f$ over $\mathcal{O}$, that is,
\[
\fint_{\mathcal{O}} f(x)\,dx = \frac{1}{|{\mathcal{O}}|}\int_{\mathcal{O}} f(x)\,dx.
\]

Concerning correctors, we have the following result.

\begin{theorem}\label{correctorlocalinballs}  Let $(u,v)\in L^2(\Omega)\times L^2(\Omega)$, with $u$ continuous in $\overline{\Omega}$, the solution of \eqref{limhlocal}-\eqref{limhnl} which satisfies 
$$\int_\Omega u(x)dx+\int_{\Omega} v(x)dx=0.$$
Considering that we have
$$0<X<1,$$
then
\begin{equation}\label{convcorrecto}
    \bigg\|u_n-\bigg(w_{1,n}-\frac{1}{2}\frac{m_n}{|A_n|}\bigg)\bigg\|_{L^2(A_n)}+ \bigg\|v_n-\bigg(
    w_{2,n}-\frac{1}{2}\frac{m_n}{|B_n|}\bigg)\bigg\|_{L^2(B_n)}\longrightarrow 0, \qquad \text{as} \quad n\to \infty.
\end{equation}
\end{theorem}

\subsection{Nonlocal in Holes}
The second case considers $B_n$ as the union of holes described by \eqref{holes}. 
In this context, a crucial analytical tool is the availability of suitable extension operators
\[
P_n : H^1(A_n) \longrightarrow H^1(\Omega),
\]
which allow us to work on the fixed domain $\Omega$ while preserving the relevant information on $A_n$.
Since the set $A_n$ is connected and the holes do not intersect the boundary $\partial\Omega$,
the existence of such operators follows from classical results on perforated domains
(see, for instance, \cite[Chapter~2.3]{doina}).
These operators satisfy uniform $H^1$--bounds, preserve the function and its gradient almost everywhere in $A_n$ and
play a fundamental role in the identification of the local term in the homogenized limit.

With this tool at hand, we can state the main homogenization result for this configuration.

\begin{theorem}[Nonlocal in Holes]\label{thrmnlinballs}
Let $(u_n,v_n)\in W_n$ be a family of solutions of \eqref{hlocal}--\eqref{hnl}, where $B_n$ is the union of holes \eqref{holes}.
Then, as $n\to \infty$, the following convergences hold:
\[
P_n u_n \to u \quad \text{strongly in } L^2(\Omega)
\qquad \text{and} \qquad
\chi_{B_n}v_n \rightharpoonup v \quad \text{weakly in } L^2(\Omega).
\]
Moreover, 
the extended functions $P_n u_n$ converge weakly to $u$ in $H^1(\Omega)$.

The limit pair $(u,v)\in H^1(\Omega)\times L^2(\Omega)$ satisfies the compatibility condition
\[
\int_\Omega Xu(x)\,dx+\int_\Omega v(x)\,dx=0,
\]
and is the unique solution of the following homogenized system:
\begin{equation}\label{nlh}
\begin{gathered}
f(x)(1-X) = \int_\Omega G(x-y)\big((1-X)v(y)-(1-X)v(x)\big)\,dy
+ \int_\Omega J(x-y)\, X\big((1-X)u(y)-v(x)\big)\,dy,
\quad x\in \Omega,
\end{gathered}
\end{equation}
\begin{equation}\label{lh}
f(x)X = \operatorname{div}(Q_{\rm Hom}\nabla u)
+ \int_{\Omega} J(x-y)\,X\big(v(y)-(1-X)u(x)\big)\,dy,
\quad x\in \Omega,
\end{equation}
with
\[
(Q_{\rm Hom}\nabla u)\cdot \eta = 0 \quad \text{on } \partial \Omega,
\]
where $\eta$ denotes the unit outward normal vector to $\partial \Omega$
and $Q_{\rm Hom}$ is the classical matrix of homogenized coefficients defined in \eqref{qij}.
\end{theorem}

The proof of the above theorem relies on the construction of suitable correctors
and on the use of the extension operators $P_n$ to handle the local diffusion term
on the varying domains $A_n$.
Due to the oscillatory nature of the perforated configuration and the presence of local/nonlocal interactions,
the convergence of $P_n u_n$ to $u$ is, in general, only weak in $H^1(\Omega)$,
which is optimal in this class of problems, since strong convergence in $H^1(\Omega)$ typically fails.

In order to improve this convergence, we also introduce a suitable corrector, this one inspired by classical
homogenization techniques, see for instance \cite{Ben, marconericardo}. More precisely, we construct a family of correctors
$(w_{1,n},w_{2,n}) \in H^1(\Omega)\times L^2(\Omega)$, such that the corrected sequence converges strongly in $H^1(A_n)\times L^2(B_n)$. This refined convergence allows us to recover strong convergence, and combining some ideas from \cite{marc2,marc3}, we set the following corrector vector 
\begin{equation}\label{corrector2}
w_n=\bigg(w_{n,1}-\frac{1}{2}\frac{m_n}{|A_n|},w_{n,2}-\frac{1}{2}\frac{m_n}{|B_n|}\bigg), %:=\bigg(u(x)-\frac{1}{n}\sum_{i=1}^ N \chi_{A_n}(x)U^i(nx)\frac{\partial u}{\partial x}(x)-\frac{1}{2}\frac{m_n}{|A_n|},\frac{\chi_{B_n}v}{1-X}(x) -\frac{1}{2}\frac{m_n}{|B_n|}\bigg)\, ,
\end{equation}
where 
\begin{equation}\label{w1nnl}
w_{n,1}(x) = u(x)-\frac{1}{n}\sum_{i=1}^ N \chi_{A_n}(x)U^i(nx)\frac{\partial u}{\partial x_i}(x)
\quad \textrm{ and }  \quad 
w_{n,2}(x) = \frac{\chi_{B_n}v}{1-X}(x),
\end{equation}
here $(u,v)\in H^1(\Omega)\times L^2(\Omega)$ is the unique solution of the homogenization equation given in
\eqref{nlh}-\eqref{lh}, $U^i$ are the auxiliary function introduced in \eqref{Ui}, 
and
$$m_n=\displaystyle\int_{A_n} w_{1,n}(x)dx+\int_{B_n}w_{2,n}(x)dx.$$
By construction, the sequence $m_n$ converges to zero as $n\to \infty.$

\begin{theorem}[Corrector for Nonlocal in Holes]\label{correctornlinballs} Assume that
$$0<X<1,$$
and let $(u,v)\in H^1(\Omega)\times L^2(\Omega)$ be a solution of the equation 
\eqref{nlh}-\eqref{lh}, then we get  
\begin{equation}\label{correctorconvergence}
\bigg\|u_n-\bigg(w_{1,n}-\frac{m_n}{2|A_n|}\bigg)\bigg\|_{H^1(A_n)}+\bigg\|v_n-\bigg(\frac{\chi_{B_n}v}{1-X}-\frac{m_n}{2|B_n|}\bigg)\bigg\|_{L^2(B_n)}\longrightarrow 0, \qquad \text{as} \quad n\to \infty.\end{equation}
\end{theorem}

\subsection{Local in Strips}
The third (and final) case appears when \eqref{local} is considered in $A_n$, where $A_n$ is now given by the family of strips in \eqref{strip}. 
This is a special situation in which we can follow the same strategy used in Theorem \ref{localinballs}, but with a fancy construction developed for the strip geometry. 
The key idea is to create a clever test function that matches this setting, allowing us to push through the homogenization argument. 
With this preparation, we are ready to state the main result, where we introduce
\[
u \in L^2([0,1]_{x_N}; H^1([0,1]^{N-1}))\]
by
\[
\begin{aligned}
u : [0,1] &\longrightarrow H^1([0,1]^{N-1}), &\quad x_N &\longmapsto u(\cdot,x_N).
\end{aligned}
\]

\begin{theorem}\label{thrmstrip} Let $(u_n,v_n) \in W_n$ be the family of solutions of \eqref{hlocal}--\eqref{hnl}, where $A_n$ is the equidistributed family of strips \eqref{strip}. 
Then the following convergences hold
$$\chi_{A_n}u_n\rightharpoonup u \in L^2([0,1]^N) \qquad \text{and} \qquad \chi_{B_n}v_n\rightharpoonup v \in L^2([0,1]^N).$$
The limit pair $(u,v)\in L^2([0,1]_{x_N};H^1([0,1]^{N-1}))\times L^2([0,1]^N)$ satisfies 
$$\displaystyle\int_{[0,1]^N} u (x) dx+\int_{[0,1]^N} v (x) dx=0,$$ with $(u,v)$ being the unique weak solution to the system:
\begin{align}\label{hstripsnl}
         f(x)(1-X)=\int_{[0,1]^N} G(x-y)&((1-X)v(y)
        -(1-X)v(x))dy\notag \\&+\int_{[0,1]^N} J(x-y)((1-X)u(y)- Xv(x))dy, \ \ x\in [0,1]^N,
\end{align}
\begin{align}\label{hstripslocal} 
    \begin{cases}&f(x)X=-\Delta_{\hat{x}} u(x)+\displaystyle\int_{[0,1]^N} J(x-y)(Xv(y)-(1-X)u(x))dy, \ \ x\in [0,1]^N,\\
    & \nabla_{\hat{x}} u\cdot \eta_{\hat{x}}=0 \quad \text{on} \; \partial [0,1]^N
\end{cases}
\end{align}
$$\Delta_{\hat{x}}:=\displaystyle\sum_{i=1}^{N-1}\frac{\partial^2}{\partial x_i^2}, \quad \nabla_{\hat{x}}:=\bigg(\displaystyle\frac{\partial}{\partial x_1},\frac{\partial}{\partial x_2}, \ldots, \frac{\partial}{\partial x_{N-1}}\bigg) \; \textrm{ and } \; \eta_{\hat{x}}:=(\eta_1,\ldots, \eta_{N-1}),$$
where $\eta$ denotes the unit outward normal vector to the boundary of $[0,1]^N$.
\end{theorem}

Notice that, as a consequence of the homogenization process, in the limit
the Laplacian has effectively "lost"
 its $x_N$-component: 
only the derivatives in the $\hat{x}$-directions survive. 
This reflects the asymptotic behavior of the solution, 
where the influence of the $x_N$-direction has been completely averaged out, 
resulting in a purely transversal diffusion. 
Observe that this is in accordance with previous results given, for instance, in \cite{GM,GM1}.

\section{Existence and uniqueness of solutions to the system \eqref{local}--\eqref{nonlocal}}
\setcounter{equation}{0}

To begin our analysis, we first establish the existence and uniqueness of solutions to the coupled local--nonlocal system \eqref{local}--\eqref{nonlocal}. 
We consider the closed Hilbert subspace
\begin{align}\label{space}
W_n \;=\;
\Bigl\{ (u,v) \in H^{1}(A_n) \times L^{2}(B_n) \;\Big|\;
\int_{A_n} u(x)\,dx \;+\; \int_{B_n} v(x)\,dx = 0
\Bigr\},
\end{align}
endowed with the bilinear form
\begin{align}\label{innerproduct}
a_n\big((u,v),(\varphi,\phi)\big)
&= \int_{A_n} \nabla u(x)\!\cdot\! \nabla \varphi(x)\,dx
  + \frac{1}{2}\!\iint_{B_n\times B_n}\!
    G(x-y)\,(v(y)-v(x))\,(\phi(y)-\phi(x))\,dx\,dy \notag\\
&\quad + \iint_{A_n\times B_n}\!
    J(x-y)\,(u(x)-v(y))\,(\varphi(x)-\phi(y))\,dx\,dy.
\end{align}
This form defines an inner product in $W_n$.

A necessary compatibility condition for the weak formulation
\[
\begin{gathered}
(u,v) \in W_n \quad \text{such that} \\
a_n\big((u,v),(\varphi,\phi)\big)
= -\int_{A_n} f(x)\,\varphi(x)\,dx
  -\int_{B_n} f(x)\,\phi(x)\,dx,
  \qquad \forall (\varphi,\phi) \in W_n,
\end{gathered}
\]
(which corresponds to the weak formulation of the system 
\eqref{local}--\eqref{nonlocal}),
is that the forcing term $f$ satisfies the compatibility condition
\[
\int_{A_n} f(x)\,dx + \int_{B_n} f(x)\,dx = 0.
\]

The system \eqref{local}--\eqref{nonlocal} can be naturally associated 
with an energy functional whose minimizer gives the unique weak solution 
in the space $W_n$. 
The uniqueness of this minimizer can also be obtained directly by 
the Lax--Milgram Theorem, showing the consistency between 
the variational and functional-analytic approaches.

For completeness, we present the details below. 
Associated with the coupled local--nonlocal problem, 
we define the energy functional $E : W_n \to \mathbb{R}$ by
\begin{align}\label{energy}
E_n(u,v)
&=\frac{1}{2}\int_{A_n} |\nabla u(x)|^2\,dx
 +\frac{1}{4}\iint_{B_n\times B_n}
   G(x-y)\,\big(v(y)-v(x)\big)^2\,dx\,dy \notag \\ &  \quad +\frac{1}{2}\iint_{A_n\times B_n}
   J(x-y)\,\big(v(y)-u(x)\big)^2\,dy\,dx
+\int_{A_n} f(x)u(x)\,dx
 +\int_{B_n} f(x)v(x)\,dx.
\end{align}
Computing the Gateaux derivative of $E$ at $(u,v)$ in the direction 
$(\varphi,\phi)\in W_n$, we obtain
\begin{align*}
\dv{}{t} E_n\big((u,v)+t(\varphi,\phi)\big)\Big|_{t=0}
&= \int_{A_n} \nabla u(x)\!\cdot\!\nabla \varphi(x)\,dx
  + \frac{1}{2}\iint_{B_n\times B_n}
     G(x-y)\,(v(y)-v(x))\,(\phi(y)-\phi(x))\,dy\,dx \\
& \quad + \iint_{A_n\times B_n}
     J(x-y)\,(v(y)-u(x))\,(\phi(y)-\varphi(x))\,dy\,dx
  \\
& \quad+ \int_{A_n} f(x)\varphi(x)\,dx
  + \int_{B_n} f(x)\phi(x)\,dx.
\end{align*}
Hence, the Euler--Lagrange equations associated with the critical points 
of $E$ are given by
\[
\partial E_n(u,v)(\varphi,\phi)=0
\qquad \forall (\varphi,\phi)\in W_n.
\]
To formally recover the Euler-Lagrange equations, we assume temporarily that $u\in H^2(A_n)$, so that Green’s identity applies to the local term. Using also the symmetry of the kernel $G$, we find
\begin{align*}
&\int_{\partial A_n} \frac{\partial u}{\partial \eta}(x)\,\varphi(x)\,dS(x)
  - \int_{A_n} \Delta u(x)\,\varphi(x)\,dx
  - \iint_{A_n\times B_n} J(x-y)\,(v(y)-u(x))\,\varphi(x)\,dy\,dx
  = -\int_{A_n} f(x)\,\varphi(x)\,dx,\\[2mm]
&-\iint_{B_n\times B_n} G(x-y)\,(v(y)-v(x))\,\phi(x)\,dy\,dx
  + \iint_{A_n\times B_n} J(x-y)\,(v(y)-u(x))\,\phi(y)\,dy\,dx
  = -\int_{B_n} f(x)\,\phi(x)\,dx.
\end{align*}
Relabeling the variables in the integral involving the kernel $J$, 
we can rewrite the system as
\begin{align*}
&\int_{\partial A_n}\frac{\partial u}{\partial \eta}(x)\,\varphi(x)\,dS(x)
 + \int_{A_n} 
   \Big[\Delta u(x)
   + \int_{B_n} J(x-y)\,(v(y)-u(x))\,dy\Big]\varphi(x)\,dx
 = \int_{A_n} f(x)\,\varphi(x)\,dx,\\[2mm]
&\int_{B_n} 
   \Big[\int_{B_n} G(x-y)\,(v(y)-v(x))\,dy
        + \int_{A_n} J(x-y)\,(u(y)-v(x))\,dy\Big]\phi(x)\,dx
 = \int_{B_n} f(x)\,\phi(x)\,dx.
\end{align*}
Thus, the weak formulation of the system \eqref{hlocal}--\eqref{hnl}
is recovered as the Euler-Lagrange equations associated with the critical points 
of the energy functional $E$.

\subsection{Existence and uniqueness}

First, we state an auxiliary lemma that provides the necessary coercivity.

\begin{lemma}\label{coerc}
There exists a constant $C_c>0$ such that
\begin{align*}
\frac{1}{2}\int_{A} |\nabla u(x)|^2\,dx 
+ \frac{1}{2}\iint_{A\times B} J(x-y)(u(x)-v(y))^2\,dx\,dy 
+ \frac{1}{4}\iint_{B\times B} G(x-y)(v(x)-v(y))^2\,dx\,dy \\
\ge C_c\left(\int_{A} u^2(x)\,dx + \int_{B} v^2(y)\,dy\right),
\end{align*}
for all $u\in H^1(A)$, $v\in L^2(B)$ with
\[
\int_{A} u(x)\,dx + \int_{B} v(x)\,dx = 0.
\]
Here, $J$ and $G$ satisfy Hypothesis \ref{hipo}.
\end{lemma}

\begin{proof}
Suppose, by contradiction, that the inequality does not hold. Then for each $n\in \mathbb{N}$ there exist $u_n \in H^1(A)$, $v_n \in L^2(B)$ such that
\[
\int_A u_n^2(x)\,dx + \int_B v_n^2(y)\,dy = 1, \qquad 
\int_A u_n(x)\,dx + \int_B v_n(y)\,dy = 0,
\]
and
\begin{align*}
\frac{1}{2}\int_A |\nabla u_n(x)|^2\,dx 
+ \frac{1}{2}\iint_{A\times B} J(x-y)(u_n(x)-v_n(y))^2\,dx\,dy 
+ \frac{1}{4}\iint_{B\times B} G(x-y)(v_n(x)-v_n(y))^2\,dx\,dy \le \frac{1}{n}.
\end{align*}

Since $(u_n)$ is bounded in $H^1(A)$ and $\int_A |\nabla u_n|^2 \to 0$, on each connected component $C\subset A$ we have, along a subsequence, $u_n \to k_C\in \mathbb{R}$ strongly in $L^2(C)$. Similarly, $(v_n)$ is bounded in $L^2(B)$, so up to a subsequence (still denoted by $v_n$), we have $v_n \rightharpoonup v$ weakly in $L^2(B)$. Using weak lower semicontinuity and recalling that both kernels $J$ and $G$ are nonnegative and the hypothesis on support of $J,G$, we have
\[
0 = \liminf_{n\to \infty} \iint_{A\times B} J(x-y)(u_n(x)-v_n(y))^2\,dy\,dx \ge \iint_{C\times B} J(x-y)(k_C-v(y))^2\,dy\,dx,
\]
and
\[
0 = \liminf_{n\to \infty} \iint_{B\times B} G(x-y)(v_n(x)-v_n(y))^2\,dx\,dy \ge \iint_{B\times B} G(x-y)(v(x)-v(y))^2\,dx\,dy.
\]
Hence, $v(x) = v(y) = c$ almost everywhere in $B$, with $c=k_C$. In addition, using our hypotheses on the supports of the kernels,
we also obtain that $v_n \to v$ strongly in $L^2(B)$. Passing to the limit in the zero mean condition gives
\[
0 = \int_A c\,dx + \int_B c\,dy = c\,(|A|+|B|) ,
\]
so $c=0$. This contradicts the normalization $\displaystyle \int_A u_n^2 + \int_B v_n^2 = 1$. Therefore, the desired inequality must hold for some constant $C_c>0$.
\end{proof}

\begin{theorem} Let $W_n\subset H^1(A_n)\times L^2(B_n)$ be the closed subspace given by \eqref{space} and assume Hypothesis \ref{hipo} for $J$ and $G$. Then, for each $f\in L^2(\Omega)$ with $\displaystyle\int_{A_n} f(x)dx+\int_{B_n} f(x)dx=0$, there exists a unique $(u,v)\in W_n$ satisfying \eqref{local}--\eqref{nonlocal} and being the minimizer of the functional \eqref{energy}.
\end{theorem}

\begin{proof} We introduce the following bilinear form $a:W_n\times W_n\to \mathbb{R}$,
\begin{align}\label{bilinearform}
    a_n((u,v),(\varphi,\phi))&=\int_{A_n} \nabla u(x)\nabla \varphi(x)dx+\iint_{A_n\times B_n} J(x-y)(v(y)-u(x))(\phi(y)-\varphi(x))dydx \notag\\+&\frac{1}{2}\iint_{B_n\times B_n} G(x-y)(v(y)-v(x))(\phi(y)-\phi(x))dydx.
\end{align}
One can check that $a_n$ is continuous, symmetric, and coercive. The coercivity can be checked as follows: first, we have that
\begin{align*}
a_n((u,v),(u,v))=\int_{A_n}|\nabla u(x)|^2dx+\iint_{A_n\times B_n} J(x-y)(u(x)-v(y))^2dxdy \\ 
\qquad \qquad \qquad +\frac{1}{2}\iint_{B_n\times B_n}G(x-y)(v(x)-v(y))^2dxdy.
\end{align*}
Now, by Lemma \eqref{coerc}, we get that
\begin{align*}
\int_{A_n}|\nabla u|^2dx+\iint_{A_n\times B_n} &J(x-y)(u(x)-v(y))^2dxdy+\iint_{B_n\times B_n}G(x-y)(v(x)-v(y))^2dxdy\geq \\
&\frac{1}{2}\int_{A_n}|\nabla u|^2dx+C_c\bigg(\int_{A_n} u^2(x)dx+\int_{B_n} v^2(x)\bigg)\geq \min{\bigg\{\frac{1}{2},C_c\bigg\}}(\|u\|_{H^1(A_n)}^2+\|v\|_{L^2(B_n)}^2),
\end{align*}
and, therefore, coercivity holds. 

Now, the linear functional
\begin{align*}
    F_n(\varphi,\phi)=\int_{A_n}\varphi(x)f(x)dx+\int_{B_n} \phi(x)f(x)dx,
\end{align*} is continuous on $W_n$, and since $\displaystyle\int_\Omega f(x)dx =0$, it follows from the Lax-Milgram Theorem that there exists a unique $(u,v)\in W_n$ satisfying
\begin{equation}\label{solution}
a_n((u,v),(\varphi,\phi))=-F_n(\varphi,\phi), \quad \quad \forall (\varphi,\phi)\in W_n.  
\end{equation}
Moreover, the pair $(u,v)$ is the unique minimizer of the energy functional.
\end{proof}

\section{Local in the Holes} \setcounter{equation}{0}

To begin the analysis of the first case, our goal is to prove Theorem \ref{localinballs}. The strategy follows the general homogenization framework, we will first establish uniform a priori estimates for the local and nonlocal components, and then use these bounds to pass to the limit in the weak sense. For this, we will prove the auxiliary lemma 
(see \cite{wolanski,marc1} for related results). Notice that in this case $B_n$ is path-connected.

\begin{lemma}[Uniform nonlocal Poincaré inequality on the perforated domains]\label{lem:unif-poincare}
Let $\Omega \subset \mathbb{R}^N$ be a bounded smooth domain and 
$B_n := \Omega \setminus \overline{T_n},$ which is path-connected.
Then there exists a constant $C_P>0$, independent of $n$, such that for every 
$v \in L^2(B_n)$
the following uniform estimate holds:
\begin{equation}\label{poincaregeneralizated}
\int_{B_n} \bigg|v(x)-\frac{1}{|B_n|}\int_{B_n} v(y)dy\bigg|^2\,dx
\;\le\;
C_P \iint_{B_n\times B_n}
G(x-y)\,\big(v(x)-v(y)\big)^2\,dy\,dx.
\end{equation}
\end{lemma}

\begin{proof}
First, we observe that each hole in $T_n$ is of the form $\frac{1}{n}(j\ell + T):=T_n(j)$; therefore, it has a diameter of order $1/n$
and is strictly contained in the scaled cell $\frac{1}{n}(j\ell + Y)$.  
Consequently, the total volume of the perforations stabilizes:
\[
|T_n| \to \frac{|T|}{|Y|}\,|\Omega|, 
\qquad
|B_n| \to \Big(1-\frac{|T|}{|Y|}\Big)|\Omega|>0.
\]
The periodic construction ensures that the total measure of the perforations does not vanish as $n\to\infty$.
Indeed, writing each periodic cell and hole as
\[
Y_{n}(j) := \tfrac{1}{n}(j\ell + Y), \qquad 
T_{n} (j):= \tfrac{1}{n}(j\ell + T),
\]
we have $|Y_{n}(j)| = |Y|/N_n$ and $|T_{n}(j)| = |T|/N_n$, here $N_n$ denotes the quantity of holes in step $n$. 
Let $I_n = \{\,j\in\mathbb{Z}^N : Y_{n}(j)\subset\Omega\,\}$ denote the set of interior cells.
Then
\[
|T_n| 
= \sum_{j\in I_n} |T_{n}(j)| 
= \frac{|T|}{N_n}|I_n|.
\]
Since $\displaystyle \frac{|I_n|}{N_n}\to \frac{|\Omega|}{|Y|}$ as $n\to\infty$, because $\displaystyle \sum_{j\in I_n}|Y_{n}(j)| = |I_n|\displaystyle\frac{|Y|}{N_n}\to |\Omega|$, we conclude that
\[
|T_n| \longrightarrow \frac{|T|}{|Y|}\,|\Omega|,
\qquad 
|B_n| = |\Omega|-|T_n| \longrightarrow 
\Big(1-\frac{|T|}{|Y|}\Big)|\Omega| > 0.
\]
Hence, although each individual hole $T_n(j)$ has diameter $\mathcal{O}(1/n)\to 0$, 
the total volume fraction of the perforations remains constant as $n\to\infty$. Thus, we can affirm there exists $\gamma\in(0,1)$ such that 
\begin{equation}\label{eq:gamma}
|B_n|\ge \gamma |\Omega|   \end{equation}
for all $n$ large enough. 

Now, we can follow the arguments in \cite[Lemma 3.1]{wolanski}. Suppose, without loss of generality, that  $\int_{B_n} v(y)dy=0$. Fix a small cube size $h>0$ and let us divide $\mathbb{R}^N$ into closed cubes $(Q_i)_{i\in\mathbb{Z}^N}$ of side length $h$.
For each $n$, consider the set of indices
\[
I_n := \{\, i : \emptyset \neq Q_i \cap B_n \subset \Omega \,\}.
\]
For $i\in I_n$, set
\[
w_i^n := |Q_i \cap B_n|,
\qquad
a_i^n := \frac{1}{w_i^n}\int_{Q_i \cap B_n} v(x)\,dx.
\]
Then $\sum_{i\in I_n} w_i^n a_i^n = 0$. Denoting by $v_h$ the piecewise constant function $v_h(x) = a_i^n$ for $x \in Q_i\cap B_n$,
one has the estimate
\begin{equation}\label{eq:L2-discrete}
\int_{B_n} |v(x)|^2\,dx
\le 2\int_{B_n}|v(x)-v_h(x)|^2\,dx
+ 2\sum_{i\in I_n} w_i^n |a_i^n|^2.
\end{equation}
By density of step functions in $L^2(B_n)$, it suffices to establish
\eqref{poincaregeneralizated} for $v=v_h$. We can write,
\begin{equation}\label{eq:variance-identity}
\sum_{i\in I_n} w_i^n |a_i^n|^2
=\frac{1}{2|B_n|}\sum_{i,k\in I_n} w_i^n w_k^n (a_i^n - a_k^n)^2.
\end{equation}
Since $G$ is continuous and nonnegative, there exist constants $c_G > 0$ and $r > 0$
such that $G(x-y) \ge c_G$ whenever $|x-y| \le r$.
Then, if the centers of $Q_i$ and $Q_k$ are at distance less than $r/2$ and using \eqref{eq:gamma} we get 
\[
w_i^n w_k^n (a_i^n - a_k^n)^2
\le \frac{1}{2c_G\, \gamma\, |\Omega|}\!
\iint_{(Q_i\cap B_n)\times(Q_k\cap B_n)}\!
G(x-y)\,(a_i^n - a_k^n)^2\,dx\,dy.
\]

The periodic geometry of $B_n$ implies the existence of an integer $L>0$,
independent of $n$, such that for any $i,k\in I_n$ there exists a chain 
$i=j_1,\dots,j_\ell=k$ with $\ell\le L$ satisfying that each $Q_{j_m}\cap B_n\neq\emptyset$;
and the centers of $Q_{j_m}$ and $Q_{j_{m+1}}$ are at distance $<r/2$.
Hence
\[
(a_i^n - a_k^n)^2
\le L\sum_{m=1}^{\ell-1} (a_{j_m}^n - a_{j_{m+1}}^n)^2.
\]
Multiplying by $w_i^n w_k^n$ and summing all pairs $(i,k)$, 
each neighboring pair of cubes appears at most $M$ times, 
where $M$ depends only on $\Omega$, $Y$, and $T$.
Therefore there exists $C_P>0$, independent of $n$, such that
\begin{align}\label{nlpoincareconst}
\sum_{i\in I_n} w_i^n |a_i^n|^2
\le
C_P \iint_{B_n\times B_n}
G(x-y)\,(a_i^n-a_k^n)^2\,dy\,dx.
\end{align}
Since $v_h\to v$ when $h\to 0$, and using \eqref{nlpoincareconst} with \eqref{eq:L2-discrete}, we obtain \eqref{poincaregeneralizated}.
\end{proof}

The first step to obtain the homogenization result consists in showing that the sequences $(u_n,v_n)$ remain uniformly bounded 
in their respective spaces. This is essential to guarantee the existence of weakly convergent 
subsequences and to identify the limit problem. In particular, by combining the structure of 
the transmission term with the uniform Poincaré-type inequality (Lemma~\ref{lem:unif-poincare}), 
we can control the $L^2$–norms of $u_n\in H^1(A_n)$ and $v_n\in L^2(B_n)$ independently of $n$. 
Once these bounds are obtained, we can ensure that the characteristic functions 
$\chi_{A_n}u_n$ and $\chi_{B_n}v_n$ admit weak limits to $u\in L^2(\Omega)$ and $v\in L^2(\Omega)$, respectively, 
which will be a key fact in the proof of Theorem~\ref{localinballs}.

\begin{lemma}\label{uvbounded}
Let $(u_n,v_n)\in W_n$ be the unique solution of \eqref{hlocal}--\eqref{hnl}. 
If the integral 
\[
\Big|\int_{B_n} v_n(x)\,dx\Big|
\]
is uniformly bounded with respect to $n$, then $u_n\in L^2(A_n)$ 
and $v_n\in L^2(B_n)$ are uniformly bounded.
\end{lemma}

\begin{proof}
Since $(u_n,v_n)\in W_n$ is a solution with 
$\big|\!\int_{B_n} v_n(x)\,dx\big|$ uniformly bounded 
and $B_n$ is connected, we can apply Lemma~\ref{lem:unif-poincare}. 
Hence,
\[
\int_{B_n} 
\Big|v_n(x)-\fint_{B_n} v_n(x)\,dx\Big|^2 dx
\le C_P
   \iint_{B_n\times B_n} 
   G(x-y)\,\big(v_n(y)-v_n(x)\big)^2\,dx\,dy 
   < C,
\]
which holds because $(u_n,v_n)$ minimizes the energy $E_n(u_n,v_n)$, 
so this quantity is uniformly bounded by a constant $C>0$ 
(independent of $n$). Consequently, 
$\|v_n\|_{L^2(B_n)}$ is uniformly bounded.

We now prove the uniform boundedness of 
\begin{equation}\label{u_nbound}
\int_{A_n} u_n^2(x)\,dx.
\end{equation}
Indeed, from the definition of the energy,
\begin{align*}
E_n(u_n,v_n)
&=\frac{1}{2}\int_{A_n} |\nabla u_n(x)|^2\,dx
+\frac{1}{4}\iint_{B_n\times B_n} 
   G(x-y)\,(v_n(y)-v_n(x))^2\,dx\,dy \\
&\quad+\frac{1}{2}\iint_{A_n\times B_n} 
   J(x-y)\,(v_n(y)-u_n(x))^2\,dy\,dx
+\int_{A_n} f_n(x)u_n(x)\,dx
+\int_{B_n} f_n(x)v_n(x)\,dx 
< C.
\end{align*}
Therefore,
\begin{align*}
\frac{1}{2}\int_{A_n} |\nabla u_n|^2\,dx
&+\frac{1}{4}\iint_{B_n\times B_n} 
   G(x-y)\,(v_n(y)-v_n(x))^2\,dx\,dy
+\frac{1}{2}\iint_{A_n\times B_n} 
   J(x-y)\,(v_n(y)-u_n(x))^2\,dy\,dx  \\
&\le C
+\Big|\int_{A_n} f_n(x)u_n(x)\,dx\Big|
+\Big|\int_{B_n} f_n(x)v_n(x)\,dx\Big|.
\end{align*}
  
Applying the arithmetic inequality 
$ab\le \frac{K^2}{2}a^2+\frac{1}{2K^2}b^2$ 
for some $K>0$ to the last term, we get
\begin{align*}
\frac{1}{2}\int_{A_n} |\nabla u_n|^2\,dx
&+\frac{1}{4}\iint_{B_n\times B_n} 
   G(x-y)\,(v_n(y)-v_n(x))^2\,dx\,dy
+\frac{1}{2}\iint_{A_n\times B_n} 
   J(x-y)\,(v_n(y)-u_n(x))^2\,dy\,dx \\
&\le C
+\Big|\int_{A_n} f_n(x)u_n(x)\,dx\Big|
+\frac{1}{2}\int_{B_n} K^2 f_n^2(x)\,dx
+\frac{1}{2}\int_{B_n} \frac{v_n^2(x)}{K^2}\,dx.
\end{align*}
To estimate the last term, we write
\begin{align*}
\bigg|\int_{B_n} v_n^2dx\bigg|
&\le \bigg|\int_{B_n} \bigg(v_n - \fint_{B_n} v_n \bigg)^2\,dx\bigg|
  + |B_n|\bigg(\,\fint_{B_n} v_n\bigg)^2 + 2\int_{B_n}v_n\,dx \fint_{B_n}v_n\\
 & \le \int_{B_n} \bigg|\bigg(v_n - \fint_{B_n} v_n \bigg)^2\bigg|\,dx
  + \frac{1}{|B_n|}\bigg(\int_{B_n}v_n\,dx\bigg)^2+\frac{2}{|B_n|}\bigg(\,\int_{B_n} v_n\,dx\bigg)^2.
\end{align*}
By the Poincaré-type inequality \eqref{poincaregeneralizated} and by the hypothesis that 
$$\bigg|\displaystyle\int_{B_n} v_n(x)\,dx\bigg|$$ is uniformly bounded
and the uniform lower bound $|B_n|\ge\gamma>0$, there exists $C$ independent of $n$, such that
\[
\int_{B_n} v_n^2\,dx
\le C_P \iint_{B_n\times B_n} 
G(x-y)\,\bigl(v_n(y)-v_n(x)\bigr)^2\,dx\,dy + C.
\]
Consequently,
\[
\frac12\int_{B_n}\frac{v_n^2(x)}{K^2}\,dx
\le \frac{C_P}{2K^2}\iint_{B_n\times B_n} 
G(x-y)\,\bigl(v_n(y)-v_n(x)\bigr)^2\,dx\,dy + \frac{C}{2K^2}.
\]
Choosing $K>0$ large enough so that $\dfrac{C_P}{2K^2}\le \dfrac18$ 
and absorbing $\dfrac{C}{2K^2}$ into the constant $C$ on the right-hand side
of the energy inequality, we obtain
\begin{align*}
\frac{1}{2}\int_{A_n} |\nabla u_n|^2\,dx
&+\frac{1}{8}\iint_{B_n\times B_n} 
   G(x-y)\,(v_n(y)-v_n(x))^2\,dx\,dy
+\frac{1}{2}\iint_{A_n\times B_n} 
   J(x-y)\,(v_n(y)-u_n(x))^2\,dy\,dx \\
&\le C
+\Big|\int_{A_n} f_n(x)u_n(x)\,dx\Big|
+\underbrace{\frac{K^2}{2}\int_{B_n} f_n^2(x)\,dx}_{\le \hat{C}}.
\end{align*}
Setting $C+\hat{C}=C$ and rearranging terms, we have
\begin{align*}
\frac{1}{8}\iint_{B_n\times B_n} 
   G(x-y)&\,(v_n(y)-v_n(x))^2\,dx\,dy
+\frac{1}{2}\iint_{A_n\times B_n} 
   J(x-y)v_n^2(y)\,dx\,dy
+\frac{1}{2}\iint_{A_n\times B_n} 
   J(x-y)u_n^2(x)\,dy\,dx \\
&-\iint_{A_n\times B_n} 
   J(x-y)v_n(y)u_n(x)\,dy\,dx
+\frac{1}{2}\int_{A_n} |\nabla u_n|^2\,dx
\le C+\Big|\int_{A_n} f_n(x)u_n(x)\,dx\Big|.
\end{align*}
Since
\[
\frac{1}{8}\iint_{B_n\times B_n} 
   G(x-y)\,(v_n(y)-v_n(x))^2\,dx\,dy \ge 0,\quad\text{and}\quad \frac{1}{2}\iint_{A_n\times B_n} 
   J(x-y)v_n^2(y)\,dx\,dy\ge 0,
\]
and
\[
b_n(x)=\int_{B_n}J(x-y)\,dy
\]
is strictly positive, there exists a constant $c_J>0$ such that 
$b_n(x)\ge c_J$. Therefore, we can rewrite this as
\begin{align*}
\frac{c_J}{2}\int_{A_n}u_n^2(x)\,dx
+\frac{1}{2}\int_{A_n} |\nabla u_n|^2\,dx
\le C+\Big|\int_{A_n} f_n(x)u_n(x)\,dx\Big|+\iint_{A_n\times B_n} 
   J(x-y)v_n(y)u_n(x)\,dy\,dx.
\end{align*}

Using the fact that $J$ is bounded, $|J(x-y)|\le C_J$,  we get 
\begin{align*}
\frac{c_J}{2}\int_{A_n}u_n^2(x)\,dx
+\frac{1}{2}\int_{A_n} |\nabla u_n|^2\,dx \le C_J\iint_{A_n\times B_n}
   \Big(\frac{v_n^2(y)}{2\tilde{k}^2}+\frac{\tilde{k}^2u_n^2(x)}{2}\Big)\,dx\,dy+\frac{1}{2}\int_{A_n}\frac{f_n^2(x)}{k^2}\,dx+\frac{1}{2}\int_{A_n}k^2u_n^2(x)\,dx,
\end{align*}
where $\tilde{k},k\in\mathbb{R}$ are constants such that 
$C_J\frac{\tilde{k}^2}{2|B_n|}=\frac{c_J}{8}$ and 
$\frac{k^2}{2}=\frac{c_J}{8}$.
Since $v_n\in L^2(B_n)$ and $f_n\in L^2(\Omega)$ are bounded, 
there exists a constant $C>0$ such that
\begin{align}\label{remarkpartialderiv}
\frac{1}{2}\int_{A_n} |\nabla u_n|^2\,dx
+\frac{c_J}{2}\int_{A_n}u_n^2(x)\,dx
\le C
+\frac{c_J}{8}\int_{A_n} u_n^2(x)\,dx
+\frac{c_J}{8}\int_{A_n}u_n^2(x)\,dx.
\end{align}
which implies
\[
\frac{c_J}{4}\int_{A_n}u_n^2(x)\,dx \le C.
\]
Therefore, $\|u_n\|_{L^2(A_n)}$ is uniformly bounded.
\end{proof}

\begin{remark}\label{stripremarkb} {\rm From \eqref{remarkpartialderiv} we also ensure that each partial derivative in 
$$\sum_{i=1}^N\int_{A_n}\bigg|\frac{\partial u_n}{\partial x_i}\bigg|^2dx$$
is uniformly bounded.}
\end{remark}

The following lemma proves that if $\displaystyle\int_{A_n}u_n(x)dx$ is uniformly bounded then we again obtain the boundedness of $v_n\in L^2(B_n)$ and $u_n\in L^2(A_n)$.
\begin{lemma}\label{uvboundf}
Let $(u_n,v_n)\in W_n$ be the unique solution of \eqref{hlocal}--\eqref{hnl}. 
If $\displaystyle\int_{B_n} v_n(y)\,dy \longrightarrow \pm\infty$, as $ n\to\infty,$
then necessarily
\[
\int_{A_n} u_n(x)\,dx \longrightarrow \mp\infty, 
\qquad \text{as } n\to\infty.
\]
\end{lemma}
\begin{proof}
By the definition of $W_n$, we have for each $n$,
\[
  \int_{A_n} u_n(x)\,dx + \int_{B_n} v_n(y)\,dy = 0.
\]
Hence
\[
  \int_{A_n} u_n(x)\,dx 
  = - \int_{B_n} v_n(y)\,dy.
\]
By hypothesis,
\[
  \int_{B_n} v_n(y)\,dy \longrightarrow \pm\infty,
\]
and therefore
\[
  \int_{A_n} u_n(x)\,dx \longrightarrow \mp\infty
\]
as $n\to\infty$. This proves the claim.
\end{proof}

Now, let us consider the bilinear form associated with \eqref{limhlocal}--\eqref{limhnl}, defined on the space
\[
\mathcal{W} := \Bigl\{ (u,v) \in L^2(\Omega) \times L^2(\Omega) \;\big|\; \int_\Omega u(x)\,dx + \int_\Omega v(x)\,dx = 0 \Bigr\},
\]
given explicitly by
\begin{align}\label{bilinearformlocalinholes}
a_\infty((u,v),(\varphi,\phi)) &=
\iint_{\Omega \times \Omega} J(x-y) \bigl[(1-X) u(x)-Xv(y) \bigr] \bigl( \phi(y) - \varphi(x) \bigr) \, dy \, dx \notag \\
&\quad + \frac{1}{2}\iint_{\Omega \times \Omega} G(x-y) \bigl[ (1-X) v(y) - (1-X) v(x) \bigr] (\phi(y)-\phi(x)) \, dy \, dx.
\end{align}
Similarly as in the previous section, by Lax-Milgram, we have that
\begin{align}\label{lmlocalinholes}
a_\infty((u,v),(\varphi,\phi))=-F_\infty(\varphi,\phi),\end{align}
with
$$F_\infty(\varphi,\phi)=\displaystyle\int_\Omega Xf\varphi(x)dx+\int_\Omega (1-X)f\phi(x) dx.$$

Since the sequences $u_n \in H^1(A_n)$ and $v_n \in L^2(B_n)$ are uniformly bounded in $L^2$, we are able to pass to the limit as $n \to \infty$.  
In this way, we obtain the homogenized bilinear form that gives the proof of Theorem~\ref{localinballs}.

\begin{proof}[Proof of the Theorem \ref{localinballs}]
We begin with the weak formulation of~\eqref{hlocal} for test functions $(\varphi_n,0)$:
\begin{equation}\label{hlocalint}
-\!\int_{A_n} f_n(x)\,\varphi_n(x)\,dx
= \int_{A_n}\nabla u_n(x)\!\cdot\!\nabla\varphi_n(x)\,dx
  - \iint_{A_n\times B_n} J(x-y)\,(v_n(y)-u_n(x))\,\varphi_n(x)\,dy\,dx,
\end{equation}
for every $\varphi_n\in C(\overline{\Omega})$ satisfying 
$\int_{A_n}\varphi_n(x)\,dx=0$.  
Rewriting \eqref{hlocalint} with the characteristic functions, we obtain
\begin{align}\label{hlocalintt}
\int_{\Omega} \chi_{A_n}(x)\,f_n(x)\,\varphi_n(x)\,dx
&= -\int_{A_n}\nabla u_n\cdot\nabla\varphi_n\,dx
   \\ & \qquad +\!\iint_{\Omega\times\Omega}
      J(x-y)\,\chi_{A_n}(x)\chi_{B_n}(y)
      (v_n(y)-u_n(x))\varphi_n(x)\,dy\,dx. \nonumber
\end{align}
The key point in this case is the choice of a family of test functions adapted to the perforated geometry. 
At this point is where the structure of the holes is crucial. Let $\varphi\in C(\overline{\Omega})$ be any given function. 
For each $n\in \mathbb{N}$, we define $\varphi_n:\Omega\to\mathbb{R}$ by
\begin{equation}\label{especialtestfunction}
\varphi_n(x)
=
\begin{cases}
\varphi(x), 
& x\in B_n,\\[0.8ex]
\displaystyle\fint_{T_n(j)} \varphi(y)\,dy, 
& x\in T_n(j)\subset A_n,
\end{cases}
\end{equation}
where $(T_n(j))_{j\le N_n}$, $ j\in\{ 1,\ldots, N_n\}$, denotes the collection of holes forming $A_n$. The functions $\varphi_n$ have the following useful properties: the first one is the uniform convergence, $\varphi_n\to\varphi$, that we show below.

Since $\varphi$ is continuous on the compact set $\overline{\Omega}$, it is uniformly continuous. 
Hence, given $\varepsilon>0$, there exists $\delta>0$ such that 
$|\varphi(x)-\varphi(y)|<\varepsilon$ whenever $|x-y|<\delta$. 
By the periodic construction, the diameter of each hole $T_n(j)$ satisfies $\mathrm{diam}(T_n(j))\to 0$ as $n\to\infty$; hence, for $n$ large enough,
\[
\mathrm{diam}(T_n(j)) < \delta
\qquad\text{for every }j.
\]
Fix $n$ sufficiently large, and let $x\in T_n(j)$. Then
\[
\Bigl|\varphi_n(x) - \varphi(x)\Bigr|
= \biggl|
   \fint_{T_n(j)}\varphi(y)\,dy - \varphi(x)
  \biggr|
\le \fint_{T_n(j)} |\varphi(y)-\varphi(x)|\,dy
< \varepsilon,
\]
since $|y-x|<\delta$ for all $y\in T_n(j)$. 

On the other hand, for $x\in B_n$ we have $\varphi_n(x)=\varphi(x)$.

Therefore, we conclude that $$\varphi_n\to\varphi$$ uniformly in $\overline{\Omega}$.

By construction, another property of $\varphi_n$ is to be constant on each hole $T_n(j)$, so
\[
\nabla\varphi_n(x) = 0 \quad\text{for all }x\in A_n.
\]
Hence, for every $n$,
\begin{equation}\label{grad-term-zero}
\int_{A_n} \nabla u_n(x)\cdot\nabla\varphi_n(x)\,dx = 0.
\end{equation}
Substituting \eqref{grad-term-zero} into \eqref{hlocalintt}, we obtain
\begin{align}\label{hlocalint-simplified}
\int_{\Omega} \chi_{A_n}(x) f_n(x)\,\varphi_n(x)\,dx
= \iint_{\Omega\times\Omega}
   J(x-y)\,\chi_{A_n}(x)\chi_{B_n}(y)\,(v_n(y)-u_n(x))\,\varphi_n(x)\,dy\,dx.
\end{align}
We now pass to the limit as $n\to\infty$ in each term of \eqref{hlocalint-simplified}.  
Since $\varphi_n\to\varphi$ uniformly in $\overline{\Omega}$ and  
$\chi_{A_n}\xrightharpoonup{*} X$ in $L^\infty(\Omega)$, we obtain
\[
\chi_{A_n}\varphi_n \longrightarrow X\varphi 
\qquad \text{strongly in } L^2(\Omega).
\]
Together with the strong convergence $f_n\to f$ in $L^2(\Omega)$, this yields
\begin{equation}\label{convf}
\int_\Omega \chi_{A_n} f_n \varphi_n\,dx 
\longrightarrow 
\int_\Omega X f(x)\varphi(x)\,dx, \quad \text{as} \ n\to \infty.
\end{equation}
For the nonlocal term, we rewrite it as
\begin{align}\label{hlocalJ}
\iint_{\Omega\times \Omega} J(x-y)\chi_{A_n}(x)\chi_{B_n}(y)(v_n(y)-u_n(x))\varphi_n(x)\,dy\,dx
&=\int_\Omega \chi_{A_n}\varphi_n \, T(\chi_{B_n}v_n)\,dx -\int_\Omega \chi_{B_n}\, T(\chi_{A_n}u_n\varphi_n)\,dy, \notag
\end{align}
where $T_J(f)(x)=\int_\Omega J(x-y)f(y)\,dy$.  
Since $J\in L^2(\Omega)$, the operator $T:L^2(\Omega)\to L^2(\Omega)$ is compact.  
From the convergences
\[
\chi_{B_n}v_n \rightharpoonup v, 
\qquad 
\chi_{A_n}u_n\varphi_n \to u\varphi
\quad\text{in }L^2(\Omega),
\]
we obtain
\[
T(\chi_{B_n}v_n)\to T(v), 
\qquad 
T(\chi_{A_n}u_n\varphi_n)\to T(u\varphi)
\quad \text{strongly in }L^2(\Omega).
\]
Since $\chi_{A_n}\varphi_n\to X\varphi$ in $L^1(\Omega)$ and  
$\chi_{B_n}\xrightharpoonup{*} (1-X)$ in $L^\infty(\Omega)$,
\begin{align}\label{convJv1}
\int_\Omega \chi_{A_n}\varphi_n\,T(\chi_{B_n}v_n)\,dx 
&\longrightarrow 
\int_\Omega X\varphi(x)\,T(v)(x)\,dx, \\
\label{convJ2}
\int_\Omega \chi_{B_n}\,T(\chi_{A_n}u_n\varphi_n)\,dy
&\longrightarrow 
\int_\Omega (1-X)(y)\,T(u\varphi)(y)\,dy. 
\end{align}
Combining \eqref{convf}, \eqref{convJv1}, \eqref{convJ2}, we obtain the limit equation corresponding to \eqref{hlocal}:
\[
\int_\Omega Xf(x)\varphi(x)\,dx
=
\int_\Omega X\varphi(x)\!\int_\Omega J(x-y)v(y)\,dy\,dx
-
\int_\Omega (1-X)\varphi(x)\!\int_\Omega J(x-y)u(x)\,dy\,dx.
\]
By the density of $C(\overline{\Omega})$ in $L^2(\Omega)$, we have
\begin{align}\label{limhlocalweak}
\int_\Omega Xf(x)\varphi(x)\,dx
=
\int_\Omega \bigg(\int_\Omega J(x-y)\,[Xv(y)-(1-X)u(x)]\,dy\bigg)\varphi(x)\,dx,
\qquad \forall \varphi\in L^2(\Omega).
\end{align}

Now, consider the weak formulation of \eqref{hnl} with the test functions $(0,\phi)$ with  
$\int_\Omega \phi=0$ that is given by
$$\int_{B_n} f\phi(x)dx=\int_{B_n}\int_{B_n}G(x-y)(v_n(y)-v_n(x))dy\phi(x)dx+\int_{B_n}\int_{A_n}J(x-y)(u_n(y)-v_n(x))dy\phi(x)dx.$$
Rewriting as before using characteristic functions and 
passing to the limit exactly as above we get
\begin{align}\label{limhnlweak}
\int_\Omega (1-X)f(x)\phi(x)\,dx
&=
\int_\Omega \bigg(\int_\Omega G(x-y)\,[(1-X)v(y)-(1-X)v(x)]\,dy\bigg)\phi(x)\,dx \\
&\quad+
\int_\Omega \bigg(\int_\Omega J(x-y)\,[(1-X)u(y)-Xv(x)]\,dy\bigg)\phi(x)\,dx, 
\qquad \forall \phi\in L^2(\Omega),\ \int_\Omega \phi=0. \notag
\end{align}
Notice that \eqref{limhlocalweak}-\eqref{limhnlweak} matches exactly with \eqref{lmlocalinholes} for all  
$(\varphi,\phi)\in L^2(\Omega)\times L^2(\Omega)$ satisfying  
$\int_\Omega\varphi + \int_\Omega\phi = 0$, then we recover the weak formulation of 
\eqref{limhlocal}--\eqref{limhnl}.

Finally, we prove that the solution to the limit system 
\eqref{limhlocal}--\eqref{limhnl} is unique. To this end, let $(u,v), (\tilde{u},\tilde{v}) \in L^2(\Omega)\times L^2(\Omega)$ 
be two solutions of \eqref{limhlocal}--\eqref{limhnl}, and define
\[
w_A := u-\tilde{u}, 
\qquad 
w_B := v-\tilde{v}.
\]
Assume that both pairs satisfy the compatibility condition
\[
\int_A u(x)\,dx + \int_B v(x)\,dx 
= \int_A \tilde{u}(x)\,dx + \int_B \tilde{v}(x)\,dx = 0.
\]
Then, subtracting the two systems \eqref{limhlocal}--\eqref{limhnl}, 
we obtain the following relations:
\begin{align}\label{usolution1}
0
&= \int_{\Omega} 
   J(x-y)\,\big[Xw_B(y) - (1-X)w_A(x)\big]\,dy,
   \qquad x\in\Omega, \\[2mm]
\label{usolution2}
0
&= \int_{\Omega} 
   G(x-y)\,\big[(1-X)w_B(y) - (1-X)w_B(x)\big]\,dy+\int_{\Omega} 
   J(x-y)\,\big[(1-X)w_A(y) - Xw_B(x)\big]\,dy,
\end{align}
for $x\in\Omega$.

Multiplying \eqref{usolution1} by $\dfrac{w_A(x)}{X}$ and 
\eqref{usolution2} by $\dfrac{w_B(x)}{1-X}$, and integrating both 
equations over $\Omega$, we find:
\begin{align*}
0
= \iint_{\Omega\times\Omega}
   J(x-y)\,
   \bigg[w_B(y)w_A(x)
         - \frac{1-X}{X}\,w_A^2(x)\bigg]\,dy\,dx,
\end{align*}
and
\begin{align*}
0= \iint_{\Omega\times\Omega}
   G(x-y)\,
   \bigg[w_B(x)w_B(y)
         - \frac{1-X}{1-X}\,w_B^2(x)\bigg]\,dy\,dx+\iint_{\Omega\times\Omega}
   J(x-y)\,
   \bigg[w_B(x)w_A(y)
         - \frac{X}{1-X}\,w_B^2(x)\bigg]\,dy\,dx.
\end{align*}

Summing the two identities and rearranging terms yields
\begin{align*}
0
&= -\iint_{\Omega\times\Omega}
     J(x-y)\,(1-X)X\,
     \bigg[\frac{w_A(x)}{X} 
           - \frac{w_B(y)}{1-X}\bigg]^2\,dy\,dx\\
&\quad + \int_{\Omega}
     \frac{w_B(x)}{1-X}
     \int_{\Omega}
     G(x-y)\,(1-X)(1-X)\,
     \bigg[\frac{w_B(y)}{1-X} 
           - \frac{w_B(x)}{1-X}\bigg]\,dy\,dx.
\end{align*}
Since the kernel 
\[
K(x-y) := G(x-y)\,(1-X)(1-X)
\]
is symmetric, we deduce the following identity for nonlocal terms, 
\begin{align*}
0
&= -\iint_{\Omega\times\Omega}
     J(x-y)\,(1-X)X\,
     \bigg[\frac{w_A(x)}{X} 
           - \frac{w_B(y)}{1-X}\bigg]^2\,dy\,dx\\
&\quad -\frac{1}{2}\iint_{\Omega\times\Omega}
     G(x-y)\,(1-X)(1-X)\,
     \bigg[\frac{w_B(y)}{1-X} 
           - \frac{w_B(x)}{1-X}\bigg]^2\,dy\,dx.
\end{align*}
Because $0 < X < 1$ for all $z\in\Omega$, each term in the right-hand side
is negative, and equality can hold only if both square brackets vanish identically.
Hence, there exists a constant $c\in\mathbb{R}$ such that
\[
\frac{w_A(x)}{X} 
= \frac{w_B(x)}{1-X} 
= c, 
\qquad \forall x\in\Omega.
\]
Therefore,
\[
w_A(x) = Xc,
\qquad
w_B(x) = (1-X)c.
\]
Since
\[
\int_A w_A(x)\,dx + \int_B w_B(x)\,dx = 0,
\]
we obtain
\[
0
= c\left(\int_A X\,dx + \int_B (1-X)\,dx\right).
\]
As $X\in(0,1)$ for all $x\in\Omega$, the quantity inside the parentheses is strictly positive,
and thus $c=0$. Consequently,
\[
w_A \equiv 0 
\quad\text{and}\quad 
w_B \equiv 0,
\]
which implies that $u=\tilde{u}$ and $v=\tilde{v}$. 
Hence, the system \eqref{limhlocal}--\eqref{limhnl} admits a unique solution.
\end{proof}

In what follows, we establish the existence of a corrector as stated in Theorem \ref{correctorlocalinballs}. Roughly speaking, by adding a suitable extra term (the corrector) to the sequence $(u_n, v_n)$, we recover strong convergence in $L^2(A) \times L^2(B)$. This follows the spirit of the nonlocal corrector approach introduced in \cite[Theorem 2.2]{monia}. 

To handle the case where the local term is defined inside the holes $A_n = T_n$, we introduce a corrector adapted to this situation. As expected, this additional term disappears in the limit, since we assume that $u\in C(\overline{\Omega})$. Recall that the approximation we consider is defined by 
\[
(w_{1,n},w_{2,n})=\left(w_{1,n},\,\frac{\chi_{B_n}v}{1-X}\right),
\]
where $w_{1,n}$ is defined in \eqref{w1n}. Here $(u,v)\in \mathcal{W}$ denotes the solution of the homogenized system \eqref{limhlocal}--\eqref{limhnl}. From the same 
computations used for the proof of \eqref{especialtestfunction}, we deduce that
\begin{equation}\label{remarkc}
w_{1,n} \longrightarrow \frac{u}{X} \quad \text{uniformly in } C(\overline{\Omega}), \quad \text{as } n\to\infty.\end{equation}
In particular, for $x\in A_n$ we obtain
\[w_{1,n}(x)
   = \chi_{T_n(j)}(x)\,\fint_{T_n(j)} \frac{u(z)}{X}\,dz 
   \;\longrightarrow\; \frac{u(x)}{X}, \quad \text{as } n\to\infty.\]
However, $(w_{1,n},w_{2,n})$ does not belong to $W_n$. To fix this, we take the adjusted vector
\[w_n:=\Bigg(w_{1,n} - \displaystyle\frac{1}{2}\frac{m_n}{|A_n|}, \;\; w_{2,n} - \displaystyle\frac{1}{2}\frac{m_n}{|B_n|}\Bigg) \in W_n,\]
with 
\[
m_n = \int_{A_n} w_{1,n}(x)\,dx + \int_{B_n} w_{2,n}(x)\,dx.
\]
Now, we are ready to proceed with the proof of Theorem \ref{correctorlocalinballs}.

\begin{proof}[Proof of Theorem \ref{correctorlocalinballs}] Consider the bilinear form
\begin{align}\label{bilinearn}
    a_n((u_n,v_n),(\varphi,\phi))&=\int_{A_n} \nabla u_n(x)\nabla \varphi(x)dx+\iint_{A_n\times B_n} J(x-y)(v_n(y)-u_n(x))(\phi(y)-\varphi(x))dydx \notag\\+&\frac{1}{2}\iint_{B_n\times B_n} G(x-y)(v_n(y)-v_n(x))(\phi(y)-\phi(x))dydx
\end{align}
already introduced in \eqref{bilinearform}, which is coercive by Lemma \eqref{coerc}. Since $\bigg(w_{1,n}-\frac{1}{2}\frac{m_n}{|A_n|},w_{2,n}-\frac{1}{2}\frac{m_n}{|B_n|}\bigg)\in W_n,$ we can insert this pair in $a_n$ and from coercivity it follows that
\begin{align}
&C_c\bigg(\bigg\|u_n-\bigg(w_{1,n}-\frac{1}{2}\frac{m_n}{|A_n|}\bigg)\bigg\|_{L^2(A_n)}^2+ \bigg\|v_n-\bigg(
    w_{2,n}-\frac{1}{2}\frac{m_n}{|B_n|}\bigg)\bigg\|^2_{L^2(B_n)}\bigg)\notag\\
    &\leq a_n\bigg(\bigg(u_n-\bigg(w_{1,n}-\frac{m_n}{2|A_n|} \bigg),v_n-\bigg(w_{2,n}-\frac{m_n}{2|B_n|}\bigg)\bigg),\bigg(u_n-\bigg(w_{1,n}-\frac{m_n}{2|A_n|} \bigg),v_n-\bigg(w_{2,n}-\frac{m_n}{2|B_n|}\bigg)\bigg)\notag\\
    &=a_n((u_n,v_n),(u_n,v_n))-2a_n\bigg((u_n,v_n),\bigg(w_{1,n}-\frac{m_n}{2|A_n|}, w_{2,n}-\frac{m_n}{2|B_n|}\bigg)\bigg)\label{line1}\\&+a_n\bigg(\bigg(\frac{m_n}{2|A_n|},\frac{m_n}{2|B_n|}\bigg),\bigg(\frac{m_n}{2|A_n|},\frac{m_n}{2|B_n|}\bigg)\bigg)- 2 a_n ((w_{1,n}, w_{2,n}) , \bigg(\frac{m_n}{2|A_n|},\frac{m_n}{2|B_n|}\bigg) \bigg)\label{line2}\\
    &+a_n((w_{1,n},w_{2,n}),(w_{1,n},w_{2,n}))\label{line3}.
\end{align}
We now compute the limit as $n \to \infty$ in these terms, in order to show that the corrector for the Neumann problem converges to the solution of 
\eqref{limhlocal}--\eqref{limhnl} in $L^2(\Omega)\times L^2(\Omega)$. Notice that, due to the way the corrector in \eqref{bilinearn} is constructed, the expression in \eqref{line2} can be reformulated as follows:
\begin{align*}&a_n\bigg(\bigg(\frac{m_n}{2|A_n|},\frac{m_n}{2|B_n|}\bigg),\bigg(\frac{m_n}{2|A_n|},\frac{m_n}{2|B_n|}\bigg)\bigg)-2a_n \bigg((w_{1,n}, w_{2,n}) , \bigg(\frac{m_n}{2|A_n|},\frac{m_n}{2|B_n|}\bigg)\bigg)\\&=\displaystyle\iint_{A_n\times B_n}J(x-y)\bigg(\frac{m_n}{2|A_n|}
 -\frac{m_n}{2|B_n|}\bigg)^2dydx \\ & \quad -2\iint_{A_n\times B_n} J(x-y)(w_{2,n}(y)-w_{1,n}(x))\bigg(\frac{m_n}{2|A_n|}-\frac{m_n}{2|B_n|}\bigg)dydx.
 \end{align*}
Replacing the integrals for $\Omega$, we get
\begin{align*}
&\iint_{\Omega\times \Omega}J(x-y)\chi_{A_n}(x)\chi_{B_n}(y)dydx\bigg(\frac{m_n}{2|A_n|}-\frac{m_n}{2|B_n|}\bigg)^2\\
&-\iint_{\Omega\times \Omega} J(x-y)\chi_{A_n}(x)\chi_{B_n}(y)(w_{2,n}(y)-w_{1,n}(x))dydx\bigg(\frac{m_n}{2|A_n|}-\frac{m_n}{2|B_n|}\bigg)\longrightarrow 
0,\qquad \text{as} \quad n\to \infty. \end{align*}
Since
\begin{align}\label{mn0}
&\frac{m_n}{2|A_n|}-\frac{m_n}{2|B_n|}=m_n\bigg(\frac{1}{2|A_n|}-\frac{1}{2|B_n|}\bigg)=\bigg(\int_{A_n}w_{1,n}dx+\int_{B_n}w_{2,n}dx\bigg)\bigg(\frac{1}{2|A_n|}-\frac{1}{2|B_n|}\bigg)\notag\\
&=\bigg(\int_{\Omega}\chi_{A_n}(x)udx+\int_{\Omega}\frac{\chi_{B_n}v}{1-X}(x)dx\bigg)\Bigg(\frac{1}{2\displaystyle\int_{\Omega}\chi_{A_n}(x)dx}-\frac{1}{2\displaystyle\int_{\Omega}\chi_{B_n}(x)dx}\Bigg)\notag\\
&\longrightarrow \bigg(\int_{\Omega}udx+\int_{\Omega}v(x)dx\bigg)\underbrace{\Bigg(\frac{1}{2\displaystyle\int_{\Omega}Xdx}-\frac{1}{2\displaystyle\int_{\Omega}(1-X)(x)dx}\Bigg)}_{bounded}=0, \qquad \text{as} \quad n\to \infty,
\end{align}
then the assumption $0=\displaystyle\int_{A_n}u_n (x) dx+\int_{B_n}v_n (x) dx$, allows us to get
\begin{equation}
\int_{A_n}u_n(x) dx+\int_{B_n}v_n (x) dx\longrightarrow \int_{\Omega}u (x) dx+\int_{\Omega}v (x) dx=0, \qquad \text{as} \quad n\to \infty.\end{equation}
The computations needed for \eqref{line2} follow from the equality \eqref{solution} and the fact that $m_n$ satisfies \eqref{mn0}, as follows:
\begin{align}
&a_n((u_n,v_n),(u_n,v_n))-2a_n\bigg((u_n,v_n),\bigg(w_{1,n}-\frac{m_n}{2|A_n|}, w_{2,n}-\frac{m_n}{2|B_n|}\bigg)\bigg)\notag
    \\
&=-(f_n,u_n)_{L^2(A_n)}-(f_n,v_n)_{L^2(B_n)}+2\bigg(f_n,w_{1,n}-\frac{m_n}{2|A_n|}\bigg)_{L^2(A_n)}+2\bigg(f_n,w_{2,n}-\frac{m_n}{2|B_n|}\bigg)_{L^2(B_n)}\notag\\
&=-\displaystyle\int_{\Omega}\chi_{A_n}u_nf_ndx-\int_{\Omega}\chi_{B_n}v_nf_ndx+2\int_{\Omega}\chi_{A_n}f_nudx+2\int_{\Omega}f_n\frac{\chi_{B_n}v}{1-X}dx-2\int_{\Omega}\chi_{A_n}f_n\frac{m_{n}}{2|A_n|}dx \notag \\ & \quad -2\int_{\Omega}\chi_{B_n}f_n\frac{m_{n}}{2|B_n|}dx \notag\\
&\longrightarrow -\int_{\Omega}ufdx-\int_{\Omega}vfdx+2\int_{\Omega}fudx+2\int_{\Omega}fvdx=\int_\Omega fudx+\int_{\Omega}fvdx, \qquad \text{as} \quad n\to \infty.\label{liml1}
\end{align}  
Now, we are able to compute the limit of \eqref{line3}, in which we must find the term $$-\int_\Omega fudx-\int_{\Omega}fvdx.$$
From \eqref{bilinearn}, that we rewrite it as follows, 
\begin{align*}
    a_n((u_n,v_n),(\varphi,\phi))&=\int_{A_n} \nabla u_n(x)\nabla \varphi(x)dx+\iint_{A_n\times B_n} J(x-y)(v_n(y)-u_n(x))(\phi(y)-\varphi(x))dydx \notag\\ & \qquad -\iint_{B_n\times B_n} G(x-y)(v_n(y)-v_n(x))\phi(x)dydx.
\end{align*}
Here we used the symmetry of $G$. 

Now, we observe that 
$$\nabla\bigg(\chi_{T_n(j)}(x) \fint_{T_n(j)} \frac{u(z)}{X(z)}dz\bigg)=0, \qquad \forall x\in \Omega$$
by the construction of the function $w_{1,n}$. 

Then, we compute the other terms,
\begin{align*}
a_n\bigg((w_{1,n},w_{2,n}),(w_{1,n},w_{2,n})\bigg)=&
\iint_{A_n\times B_n} J(x-y)(w_{2,n}(y)-w_{1,n}(x))(w_{2,n}(y)-w_{1,n}(x))dydx \notag\\-&\iint_{B_n\times B_n} G(x-y)(w_{2,n}(y)-w_{2,n}(x))w_{2,n}(x)dydx,
\end{align*}
where we can rewrite this as follows,
\begin{align}
 a_n \bigg((w_{1,n},w_{2,n})&,(w_{1,n},w_{2,n})\bigg)=\iint_{A_n\times B_n}J(x-y)\bigg(\frac{\chi_{B_n}v(y)}{1-X}-\chi_{T_n(j)}(x) \fint_{T_n(j)} \frac{u(z)}{X(z)}dz\bigg)\frac{\chi_{B_n}v(y)}{1-X}dydx\label{l1}\\
 &-\iint_{A_n\times B_n}J(x-y)\bigg(\frac{\chi_{B_n}v(y)}{1-X}-\chi_{T_n(j)}(x) \fint_{T_n(j)} \frac{u(z)}{X(z)}dz\bigg)\chi_{T_n(j)}(x) \fint_{T_n(j)} \frac{u(z)}{X(z)}dzdydx\label{l2}\\
 &-\iint_{B_n\times B_n} G(x-y)\bigg(\frac{\chi_{B_n}v(y)}{1-X}-\frac{\chi_{B_n}v(x)}{1-X}\bigg)\frac{\chi_{B_n}v(x)}{1-X}dydx. \label{l3}
\end{align}

We treat each term separately. For \eqref{l1}, we rewrite the integral using characteristic functions and the identity
$\chi_{B_n}^2=\chi_{B_{n}}$ and we obtain,
\begin{align}
&\iint_{A_n\times B_n}J(x-y)\bigg(\frac{\chi_{B_n}v(y)}{1-X}-\chi_{T_n(j)}(x) \fint_{T_n(j)} \frac{u(z)}{X(z)}dz\bigg)\frac{\chi_{B_n}v(y)}{1-X}dydx=\notag\\
&\int_\Omega\frac{\chi_{B_n}v(y)}{(1-X)^2(y)}\int_\Omega J(x-y)\chi_{A_n}(x)v(y)dydx\label{1l1}\\
&-\int_{\Omega}\chi_{T_n(j)}(x) \fint_{T_n(j)} \frac{u(z)}{X(z)}dz\int_{\Omega}J(x-y)\chi_{A_n}(x)\frac{\chi_{B_n}v(y)}{1-X}dydx.\label{2l1}
\end{align}
For the first term \eqref{1l1}, since $\chi_{A_n}\xrightharpoonup{*} X$ in $L^\infty(\Omega)$ and 
$y\mapsto J(x-y)$ is continuous, we have
\[
\int_\Omega J(x-y)\chi_{A_n}(x)\,dx
    \longrightarrow \int_\Omega J(x-y)X(x)\,dx
    \qquad\text{in }L^\infty(\Omega).
\]
From $\chi_{B_n}\xrightharpoonup{*}1-X$ and the fact that
\[
f(y)=\frac{v^2(y)}{(1-X)^2(y)}\int_\Omega J(x-y)X(x)\,dx \in L^1(\Omega),
\]
we obtain
\begin{equation}\label{1line1}
\int_\Omega\frac{\chi_{B_n}v(y)}{(1-X)^2(y)}
     \left(\int_\Omega J(x-y)\chi_{A_n}(x)\,dx\right)v(y)\,dy
    \longrightarrow
    \int_\Omega\frac{v(y)}{1-X(y)}
        \left(\int_\Omega J(x-y)X(x)\,dx\right)v(y)\,dy.
\end{equation}

For the second term \eqref{2l1}, the uniform convergence \eqref{remarkc} yields
\[
\chi_{T_n(j)}(x)\fint_{T_n(j)}\frac{u}{X}\,dz \longrightarrow \frac{u(x)}{X(x)}
    \qquad\text{uniformly},
\]
and, since
\[
\int_\Omega J(x-y)\frac{\chi_{B_n}v(y)}{1-X(y)}\,dy
     \longrightarrow
     \int_\Omega J(x-y)v(y)\,dy \qquad\text{in }L^2(\Omega),
\]
we conclude
\begin{equation}\label{2line1}
\begin{array}{l}
\displaystyle 
-\!\int_{\Omega}\chi_{A_n}(x)
  \left(\chi_{T_n(j)}(x)\fint_{T_n(j)}\frac{u}{X}\,dz\right)
  \left(\int_\Omega J(x-y)\frac{\chi_{B_n}v(y)}{1-X(y)}\,dy\right)dx \\[8pt]
  \displaystyle \qquad 
\longrightarrow
-\!\int_\Omega\frac{u(x)}{X(x)}
     \left(\int_\Omega J(x-y)X(y)v(y)\,dy\right)dx.
     \end{array}
\end{equation}

We now turn out attention  to \eqref{l2}. Again writing integrals using characteristic functions gives
\begin{align}
&-\iint_{A_n\times B_n}J(x-y)\bigg(\frac{\chi_{B_n}v(y)}{1-X}-\chi_{T_n(j)}(x) \fint_{T_n(j)} \frac{u(z)}{X(z)}dz\bigg)\bigg(\chi_{T_n(j)}(x) \fint_{T_n(j)} \frac{u(z)}{X(z)}dz\bigg)dydx=\notag\\
&\int_\Omega\bigg(\chi_{T_n(j)}(x) \fint_{T_n(j)} \frac{u(z)}{X(z)}dz\bigg)\int_\Omega J(x-y)\chi_{B_n}(y)dy\chi_{A_n}(x)\bigg(\chi_{T_n(j)}(x) \fint_{T_n(j)} \frac{u(z)}{X(z)}dz\bigg)dx\label{1l2}
\\
&-\int_\Omega\chi_{B_n}(y)\frac{v}{1-X}(y)\int_\Omega J(x-y)\chi_{A_n}(x)\bigg(\chi_{T_n(j)}(x) \fint_{T_n(j)} \frac{u(z)}{X(z)}dz\bigg)dxdy.\label{2l2}
\end{align}
In the first term \eqref{1l2}, we pass to the limit as follows,
\begin{align}\label{lim1l2}\lim_{n\to \infty}&\int_\Omega\bigg(\chi_{T_n(j)}(x) \fint_{T_n(j)} \frac{u(z)}{X(z)}dz\bigg)\int_\Omega J(x-y)\chi_{B_n}(y)dy\chi_{A_n}(x)\bigg(\chi_{T_n(j)}(x) \fint_{T_n(j)} \frac{u(z)}{X(z)}dz\bigg)dx=\notag\\
&\int_{\Omega} \frac{u}{X}(x)\int_{\Omega}J(x-y)(1-X)(y)dyu(x)dx,
\end{align}
Indeed, since  $\chi_{B_n}\xrightharpoonup{*}1-X $  
and the function $x\mapsto J(x-y)$ is continuous, we have the following limit,
$$\int_\Omega J(x-y)\chi_{B_n}(y)dy\longrightarrow \int_{\Omega} J(x-y)(1-X)dy\in L^\infty(\Omega),\qquad \text{as} \quad n\to \infty.$$
Since
$w_{n,1}\to \displaystyle\frac{u}{X}$ uniformly by \eqref{remarkc} and $\chi_{A_n}\xrightharpoonup{*} X$, then the product also converges  $\chi_{A_n}(w_{1,n})^2\to \displaystyle\frac{u^2}{X} \in L^2(\Omega)$, and we obtain the limit given by \eqref{lim1l2}. For \eqref{2l2}, we compute the limit, as $n\to \infty$, and we obtain
\begin{align}
& \lim_{n\to \infty}-\int_\Omega\chi_{B_n}(y)\frac{v}{1-X}(y)\int_\Omega J(x-y)\chi_{A_n}(x)\bigg(\chi_{T_n(j)}(x) \fint_{T_n(j)} \frac{u(z)}{X(z)}dz\bigg)dxdy\label{liml2}\\
&=-\int_\Omega\frac{v}{1-X}(y)\int_\Omega J(x-y)[(1-Xu(x)]dydx.\notag
\end{align}
Since $\chi_{A_n}w_{n,1}\rightharpoonup u\in L^2(\Omega)$ and $y\mapsto J(x-y)$ is a continuous function, then
$$\int_\Omega J(x-y)\chi_{A_n}(x)\bigg(\chi_{T_n(j)}(x) \fint_{T_n(j)} \frac{u(z)}{X(z)}dz\bigg)dx\longrightarrow \int_\Omega J(x-y)u(x)dx,\qquad \text{as} \quad n\to \infty.$$
The limit \eqref{liml2} follows from the fact that $\chi_{B_n}\xrightharpoonup{*}1-X$. To verify the last line \eqref{l3}, we first find the following expression,
\begin{align*}
    &-\iint_{B_n\times B_n} G(x-y)\bigg(\frac{\chi_{B_n}v(y)}{1-X}-\frac{\chi_{B_n}v(x)}{1-X}\bigg)\frac{\chi_{B_n}v(x)}{1-X}dydx\\
    &=\int_\Omega\frac{v(x)}{1-X}\int_\Omega G(x-y)\chi_{B_n}(y)\frac{\chi_{B_n}v(x)}{1-X}dydx-\int_{\Omega}\frac{v(x)}{1-X}\int_{\Omega} G(x-y)\chi_{B_n}(x)\frac{\chi_{B_n}v(y)}{1-X}dydx.
\end{align*}
Since $\chi_{B_n}\xrightharpoonup{*}1-X \in L^{\infty}(\Omega)$ and repeating the same weak convergence arguments invoking the continuity of $G$, we compute the following limit
\begin{align}\label{1l3}
   & \lim_{n\to \infty}\int_\Omega\frac{v(x)}{1-X}\int_\Omega G(x-y)\chi_{B_n}(y)\frac{\chi_{B_n}v(x)}{1-X}dydx-\int_{\Omega}\frac{v(x)}{1-(X)}\int_{\Omega} G(x-y)\chi_{B_n}(x)\frac{\chi_{B_n}v(y)}{1-X}dydx\\
    &=\int_\Omega\frac{v(x)}{1-X}\int_\Omega G(x-y)(1-X)v(x)dydx-\int_{\Omega}\frac{v(x)}{1-X}\int_{\Omega} G(x-y)(1-X)v(y)dydx\notag\\
    &=\int_\Omega \frac{v(x)}{1-X}\int_\Omega G(x-y)[(1-X)v(x)-(1-X)v(y)]dydx.\notag
\end{align}
At this stage, gathering the previous limits of \eqref{1l1}, \eqref{2l1},\eqref{1l2},\eqref{2l2} and \eqref{1l3} we conclude that 
\begin{align*}
&\lim_{n\to \infty}a_n((w_{1,n},w_{2,n}),(w_{1,n},w_{2,n}))=\int_{\Omega} \frac{u}{X}(x)\int_{\Omega}J(x-y)((1-X)(y)dyu(x)dx-Xv(y))dydx\\&\qquad +\int_{\Omega} \frac{v(y)}{1-X}\int_{\Omega} J(x-y)(Xv(y)-(1-X)u(x)]dydx\\& \qquad +\int_\Omega \frac{v(x)}{1-X}\int_\Omega G(x-y)[(1-X)v(x)-(1-X)v(y)]dydx,
\end{align*}
which gives us the formulation of \eqref{limhlocal}--\eqref{limhnl} if we multiply by $\frac{u}{X}$ and $\frac{v}{1-X}$ and integrate both equations inside $\Omega$. Then, we conclude that 
\begin{align*}
    \lim_{n\to \infty}a_n((w_{1,n},w_{2,n}),(w_{1,n},w_{2,n}))=-\int_{\Omega} ufdx -\int_{\Omega}vfdx,
\end{align*}
which proves the theorem.
\end{proof}

\section{NonLocal in Holes} \setcounter{equation}{0}
We now move to the complementary configuration.   
Precisely, we describe $A_n$, $B_n$ as  
\begin{equation}\label{nlinball}
    B_n = \bigcup_{j=1}^{N_n} T_n(j),
    \qquad 
    A_n = \Omega \setminus B_n.
\end{equation}
To analyze the limit as $n \to \infty$ of the sequence of solutions $(u_n, v_n)\in W_n$ to the system \eqref{hlocal}--\eqref{hnl}, we must take into account the new geometric configuration of the perforated domain. In particular, the Neumann boundary condition is imposed both on the internal boundary of the holes, $\partial_{\mathrm{int}}\Omega$, and on the external boundary $\partial_{\mathrm{ext}}\Omega$.

Our approach is based on Theorem~\ref{solution}, and the main convergence result will be stated in Theorem~\ref{thrmnlinballs}. Before studying the asymptotic behavior, we recall some auxiliary tools that will be used in the analysis, for the proofs we refer to \cite{doina, luctartar}.

The homogenized problem \eqref{nlh}--\eqref{lh} associated with \eqref{hlocal}--\eqref{hnl} is formulated in terms of a coercive bilinear form on 
\[
\mathcal{W} := \bigg\{ (u,v) \in  H^1(\Omega) \times L^2(\Omega) \;\big|\; \int_{\Omega} Xu(x)\,dx + \int_{\Omega} v(x)\,dx = 0 \bigg\},
\]
that is explicitly given by
\begin{align}\label{wfnlh}
a_\infty((u,v),(\varphi,\phi)) 
&= \int_\Omega Q_{\rm Hom} \nabla u \cdot \nabla \varphi(x) \, dx + \iint_{\Omega \times \Omega} J(x-y) \, X\bigl(v(y) - (1-X) u(x) \bigr) \bigl(\phi(y) - \varphi(x) \bigl) \, dy \, dx \notag \\
&\quad + \frac12\iint_{\Omega \times \Omega} G(x-y) \bigl((1-X)v(y) -  (1-X)v(x) \bigr)(\phi(y)-\phi(x)) \, dy \, dx, 
\end{align}
where $Q_{\rm Hom}$ is the following positive definite matrix 
\begin{equation}\label{qij}
    Q_{\rm Hom}:=\frac{1}{|Y|}\int_{Y^*}I-\bigg[\frac{\partial U^i}{\partial y_j}\bigg]_{i,j=1}^Ndx
\end{equation}
with
\begin{equation}
\bigg[\frac{\partial U^i}{\partial y_j}\bigg]_{j,i=1}^N:=\begin{pmatrix}
\frac{\partial U^1}{\partial y_1} & \frac{\partial U^2}{\partial y_1} & \cdots & \frac{\partial U^N}{\partial y_1} \\
\frac{\partial U^1}{\partial y_2} & \frac{\partial U^2}{\partial y_2} & \cdots & \frac{\partial U^N}{\partial y_2} \\
\vdots & \vdots & \ddots & \vdots \\
\frac{\partial U^1}{\partial y_N} & \frac{\partial U^2}{\partial y_N} & \cdots & \frac{\partial U^N}{\partial y_N}
\end{pmatrix}.
\end{equation}
The auxiliar function $U^i(x):Y^*\to \mathbb{R}$, $i=1,\ldots N$, is the unique solution of the harmonic problem
\begin{equation}\label{Ui}
\begin{cases}
    -\Delta U^i(x)=0 \ \text{in}\ Y^*\\
    \nabla U^i\cdot \eta =\eta_i(x) \ \text{in}\ \partial_{int}Y^*\\
    U^i(x) \ \ \ \  Y\text{-periodic}\
\end{cases}\end{equation}
in the representative cell $Y^*$ introduced in \eqref{cell}. 

Then $a_\infty$ define a coercive bilinear form such that solutions $(u,v)$ to the homenized problem satisfy
\begin{equation}\label{lmnlinholes}
a_\infty((u,v),(\varphi,\phi))=-\int_{\Omega} fX\varphi(x)dx-\int_\Omega f(1-X)\phi(x)dx.\end{equation}

\begin{remark} \label{Q_hom} {\rm For the sake of completeness, we show that the matrix $Q_{\rm Hom}$ is positive definite. 
Recall that the coefficients of $Q_{\rm Hom}$ are given by
\begin{equation}\label{Q}
Q_{\rm Hom}:=\frac{1}{|Y|}\int_{Y^*}I-\bigg[\frac{\partial U^i}{\partial y_j}\bigg]_{i,j=1}^Ndx,\end{equation}
where $U^i$ is given by the equation \eqref{Ui}, which variational formulation correspond is, to find $U^i$ $Y-$ periodic and
\begin{equation}\label{bilinearU}\int_{Y^*} \nabla U^i \cdot \nabla \varphi \, dy = \int_{\partial Y^*} \eta_i \varphi \, dS \quad \text{for all } \varphi\in H^1_{per}(Y^*).\end{equation}
The goal is to show that for any nonzero vector $\boldsymbol{\xi} = (\xi_1, \ldots, \xi_N) \in \mathbb{R}^N \setminus \{\mathbf{0}\}$,
\[
\sum_{i,j=1}^N q_{ij} \xi_i \xi_j > 0,
\]
where we are denoting $q_{ij}$ the coefficients of the matrix $Q_{\rm Hom}$,
which are 
\begin{equation}\label{coefficient}
q_{ij}:=\frac{1}{|Y|}\int_{Y^*}\delta_{ij}-\frac{\partial U^i}{\partial y_j}dx.\end{equation}
For a given vector $\boldsymbol{\xi}$, define the function
$$w = \sum_{i=1}^N \xi_i (y_i - U^i),$$
whose gradient is given by
$$\nabla w = \sum_{i=1}^N \xi_i (\mathbf{e}_i - \nabla U^i).$$
Then, we can compute the energy of $w$,
\begin{align}
\int_{Y^*} |\nabla w|^2 dy &= \int_{Y^*} \left( \sum_{i=1}^N \xi_i (\mathbf{e}_i - \nabla U^i) \right) \cdot \left( \sum_{j=1}^N \xi_j (\mathbf{e}_j - \nabla U^j) \right) dy \notag\\
&= \sum_{i,j} \xi_i \xi_j \int_{Y^*} (\mathbf{e}_i - \nabla U^i) \cdot (\mathbf{e}_j - \nabla U^j) \, dy. \label{eq:energy_expansion}
\end{align}
We have
\[
(\mathbf{e}_i - \nabla U^i) \cdot (\mathbf{e}_j - \nabla U^j) = \delta_{ij} - \frac{\partial U^j}{\partial y_i} - \frac{\partial U^i}{\partial y_j} + \nabla U^i \cdot \nabla U^j.
\]
Going back into \eqref{eq:energy_expansion}, we get
\begin{equation}
\int_{Y^*} |\nabla w|^2 dy = \sum_{i,j} \xi_i \xi_j \int_{Y^*} \left( \delta_{ij} - \frac{\partial U^j}{\partial y_i} - \frac{\partial U^i}{\partial y_j} + \nabla U^i \cdot \nabla U^j \right) dy. \label{eq:energy_full}
\end{equation}
We now simplify the terms in \eqref{eq:energy_full} using the weak form of the equations and the divergence theorem. First, the weak form \eqref{bilinearU} with the test function $\varphi = U^j$ gives us
\begin{align*}
    0=-\int_{Y^*}U^j(\Delta U^i)&=\int_{Y^*}\nabla U^i(y)\nabla U^j(y)dy-\int_{\partial Y^*}U^j(y)\underbrace{\frac{\partial U^i}{\partial \eta}}_{=\eta_i}(y)dS(y)\\
    &=\int_{Y^*}\nabla U^i(y)
    \nabla U^j(y)dy-\int_{\partial Y^*}[U^j\mathbf{e}_i](y)\eta dS_i(y)\\
    &=\int_{Y^*}\nabla U^i(y)
    \nabla U^j(y)dy-\int_{ Y^*}div [U^j\mathbf{e}_i](y)dy\\
    &=\int_{Y^*}\nabla U^i(y)\nabla U^j(y)dy-\int_{Y^*}\frac{\partial U^j}{\partial y_i}dy,
\end{align*}
here we used the Divergence Theorem. We can conclude that
\begin{align}\label{relationint}
\int_{Y^*}\nabla U^i(y)\nabla U^j(y)dy=\int_{Y^*}\frac{\partial U^j}{\partial y_i}dy=\int_{Y^*}\frac{\partial U^i}{\partial y_j}dy.\end{align}
Now, substitute the above into the expression for the energy \eqref{eq:energy_full} to obtain
\begin{align}\label{eq:energy_simple} 
\int_{Y^*} |\nabla w|^2 dy &= \sum_{i,j} \xi_i \xi_j \bigg[ \int_{Y^*} \delta_{ij} dy - \int_{Y^*}\frac{\partial U^j}{\partial y_i}dy -\int_{Y^*}\frac{\partial U^i}{\partial y_j}dy +\int_{Y^*}\frac{\partial U^j}{\partial y_i}dy\bigg] \notag\\
&= \sum_{i,j} \xi_i \xi_j \bigg[ \delta_{ij} |Y^*| -\int_{Y^*}\frac{\partial U^i}{\partial y_j}dy \bigg]. 
\end{align}
From the definition of $q_{ij}$ in \eqref{coefficient} it follows that
\[
|Y| \sum_{i,j} q_{ij} \xi_i \xi_j = \sum_{i,j} \xi_i \xi_j \int_{Y^*} \left( \delta_{ij} - \frac{\partial U^i}{\partial y_j} \right) dy.
\]
Using \eqref{relationint} and integration by parts in \eqref{Ui}, we get
\begin{equation}\label{eq:q_identity}
|Y| \sum_{i,j} q_{ij} \xi_i \xi_j = \sum_{i,j} \xi_i \xi_j \left( \delta_{ij} |Y^*| -\underbrace{\int_{Y^*} \nabla U^i \eta_j dS}_{=0}\right). 
\end{equation}
Comparing \eqref{eq:energy_simple} and \eqref{eq:q_identity}, we obtain the key identity
\begin{equation}
\int_{Y^*} |\nabla w|^2 dy = |Y| \sum_{i,j} q_{ij} \xi_i \xi_j. \label{eq:main_identity}
\end{equation}
Now, from \eqref{eq:main_identity}, we conclude that
\[
\sum_{i,j} q_{ij} \xi_i \xi_j = \frac{1}{|Y|} \int_{Y^*} |\nabla w|^2 dy.
\]
Since $|Y| > 0$ and the energy integral is non-negative, we have
\[
\sum_{i,j} q_{ij} \xi_i \xi_j \geq 0.
\]
Now, suppose $\sum_{i,j} q_{ij} \xi_i \xi_j = 0$ for some vector $\boldsymbol{\xi}$. Then
\[
\int_{Y^*} |\nabla w|^2 dy = 0,
\]
which implies $\nabla w = 0$ almost everywhere in $Y^*$. Therefore, $w$ is constant in $Y^*$. Recall that
\[
w = \sum_{i=1}^N \xi_i (y_i - U^i).
\]
The function $y_i$ is not periodic, while $U^i$ is $Y$-periodic. For $w$ to be constant, the coefficient $\boldsymbol{\xi}$ must be zero. Hence, if $\boldsymbol{\xi} \neq \mathbf{0}$, then
\[
\sum_{i,j} q_{ij} \xi_i \xi_j > 0,
\]
which proves that the matrix $Q_{\rm Hom}$ is positive definite. }
\end{remark}

Now, we return to the proof of the homogenization result. In this case, corresponding to \eqref{nlinball}, the connectedness of the set $A_n$ plays a crucial role. 
It allows us to use a linear \emph{extension operator}
\[
P_n : H^1(A_n) \longrightarrow H^1(\Omega),
\]
whose existence follows from classical results on perforated domains (see, for instance, \cite{doina2}). 
This operator enables us to work on the fixed domain $\Omega$ while keeping uniform control of functions originally defined on $A_n$.
More precisely, there exists a constant $C>0$, depending only on $\Omega$ (independent of $n$), such that
\begin{equation}\label{extension}
\|P_n u\|_{H^1(\Omega)} \le C\,\|u\|_{H^1(A_n)},
\qquad \forall\,u\in H^1(A_n).
\end{equation}
Moreover, the extension operator preserves the function and its gradient almost everywhere in $A_n$, namely,
\[
P_n u = u \quad \text{and} \quad \nabla P_n u = \nabla u \quad \text{in } A_n.
\]

We also introduce the notation $\tilde{u}$ to denote the extension of a function $u$ in $A_n$ by zero to the whole domain $\Omega$.

Unlike in the complementary configuration, here a local second order operator remains in the limit. 
The existence of the operator $P_n$ allows us not only to express the integrals over $A_n$  as integrals over the fixed domain $\Omega$, 
but also to pass to the limit in these integrals when analyzing the associated bilinear form given by~\eqref{bilinearform}.  
The boundedness of $u_n\in L^2(A_n)$ and $v_n\in L^2(B_n)$ can be justified by the same reasoning as in Lemma~\ref{uvbounded}.  
Indeed, applying the classical Poincaré--Wirtinger inequality for periodic holes (see, for instance, \cite[Lemma 2.1]{doina}), 
we immediately obtain uniform bound on the corresponding norms, which ensures that the energy estimates stay valid in this setting.

\begin{lemma}\label{vubounded}
Let $(u_n,v_n)\in W_n$ be the unique solution of \eqref{hlocal}--\eqref{hnl}. 
If $\displaystyle\Big|\!\int_{A_n} u_n(x)\,dx\Big|$ is uniformly bounded in $n$, 
then the sequences 
$u_n,\ \partial_{x_i}u_n\in L^2(A_n)$ for $i=1,\ldots,N$, 
and $v_n\in L^2(B_n)$ are uniformly bounded.
\end{lemma}

\begin{proof}
Since $A_n$ is connected and the mean value of $u_n$ over $A_n$ is uniformly bounded, 
the classical Poincaré--Wirtinger inequality applies with a constant $C_P>0$ 
independent of $n$:
\[
\int_{A_n}\Big|u_n(x)-\fint_{A_n}u_n(x)\,dx\Big|^2dx
\le C_P \int_{A_n}|\nabla u_n(x)|^2dx.
\]
Because $(u_n,v_n)$ minimizes the energy $E_n(u_n,v_n)$, 
the right-hand side is uniformly bounded by some constant $C>0$, 
and consequently $\|u_n\|_{L^2(A_n)}$ is uniformly bounded.  
The same reasoning, applied to $v_n$ (using the structure of $E$ and the estimates 
previously obtained in Lemma~\ref{u_nbound}), 
gives the uniform bound for $\|v_n\|_{L^2(B_n)}$.
\end{proof}

By the extension estimate~\eqref{extension}, the sequence $\{P_n u_n\}$ is uniformly bounded in $H^1(\Omega)$, and hence in $L^2(\Omega)$. 
Therefore, up to a subsequence, there exist functions $u\in H^1(\Omega)$ and $v\in L^2(\Omega)$ such that
\[
P_n u_n \rightharpoonup u \quad\text{in } H^1(\Omega),
\qquad
\chi_{B_n}v_n \rightharpoonup v \quad\text{in } L^2(\Omega),
\qquad \text{as } n\to\infty.
\]
Moreover, by the Rellich--Kondrachov compactness theorem, we obtain the strong convergence
\[
P_n u_n \longrightarrow u 
\quad \text{in } L^2(\Omega).
\]
Since $\chi_{A_n}\xrightharpoonup{*} X$ in $L^{\infty}(\Omega)$, it follows that
\[
\chi_{A_n}P_n u_n
=\tilde{u}_n
\rightharpoonup X u 
\quad \text{in } L^2(\Omega),
\qquad \text{as } n\to\infty.
\]

As before, $(u,v)\in H^1
(\Omega)\times L^2(\Omega)$ denotes the unique weak solution of the homogenized problem given by the weak formulation~\eqref{lmnlinholes}. In this setting, it is important to note that the convergence $P_n u_n \rightharpoonup u$ in $H^1(\Omega)$, established earlier, is only weak. In fact, this is the optimal type of convergence one may expect for problems with periodic structures; strong convergence in $H^1$ generally fails, as we mentioned before.

To obtain a stronger type of convergence, we therefore introduce a suitable corrector procedure. The idea is to modify the homogenized solution $(u,v)$ by adding appropriate corrector functions $(w_{n,1},w_{n,2})$, allowing a form of strong approximation of $P_n u_n$, given by
\begin{equation}
w_n=(w_{n,1},w_{n,2})=\bigg(u(x)-\frac{1}{n}\sum_{i=1}^ N \chi_{A_n}(x)U^i(nx)\frac{\partial u}{\partial x_i}(x),\frac{\chi_{B_n}v}{1-X}(x) \bigg)\,\in H^1(A_n)\times L^2(B_n),\end{equation}
where $(u,v)$ is the unique weak solution of
\eqref{nlh}--\eqref{lh}. Since we have Neumann boundary conditions, we also need to guarantee that the vector must be in $W_n$. 
For this, we set 
\[
\bigg(w_{1,n}-\frac{1}{2}\frac{m_n}{|A_n|},\, w_{2,n}-\frac{1}{2}\frac{m_n}{|B_n|}\bigg) \in W_n,
\]
with 
\[
m_n=\int_{A_n} w_{1,n}(x)\,dx+\int_{B_n} w_{2,n}(x)\,dx,
\]
which is exactly the correction term ensuring the strong convergence in $H^1(\Omega)\times L^2(\Omega)$ of the solution $(u_n,v_n)\in W_n$ to its homogenized limit $(u,v)\in \mathcal{W}$ and vanish in the limit. Moreover, in Theorem~\ref{thrmnlinballs} we establish that this limit is the unique solution in $H^1(\Omega)\times L^2(\Omega)$.
 
\begin{proof}[Proof of Theorem \ref{correctornlinballs}]
The bilinear form \eqref{bilinearform} corresponding to the system \eqref{hlocal}--\eqref{hnl} is coercive, by Proposition \ref{coerc}. Then, there exists a constant $C_c>0$ such that $$C_c(\|u_n\|_{L^2(A_n)}+\|v_n\|_{L^2(B_n)})\leq \frac12 a_n((u_n,v_n),(u_n,v_n)).$$
Moreover, recalling the explicit structure of $a_n$, and keeping only a portion of the Laplacian contribution, we also obtain a coercive estimate
$$a_n\bigl((u_n,v_n),(u_n,v_n)\bigr)
   \;\geq\; \alpha \|\nabla u_n\|_{L^2(A_n)}^2
            + C_1\bigl(\|u_n\|_{L^2(A_n)}^2 + \|v_n\|_{L^2(B_n)}^2\bigr)=K(\|u_n\|_{H^1(A_n)}+\|v_n\|_{L^2(B_n)}),$$
for some $\alpha,C_1,K>0$. 

Consider the vector $$\Big( w_{1,n}-\frac{1}{2}\frac{m_n}{|A_n|},w_{2,n}-\frac{1}{2}\frac{m_n}{|B_n|}\Big)\in W_n$$ defined in \eqref{corrector2}. We have,
\begin{align}
    &K\bigg(\bigg\|u_n-\bigg(w_{n,1}-\frac{m_n}{2|A_n|}\bigg)\bigg\|_{H^1(A_n)}+\bigg\|v_n-\bigg(\frac{\chi_{B_n}v}{1-X}-\frac{m_n}{2|B_n|}\bigg)\bigg\|_{L^2(B_n)}\bigg)\notag\\&
    \leq a_n\bigg(\bigg(u_n-\bigg(w_{1,n}-\frac{m_n}{2|A_n|} \bigg)\bigg),v_n-\bigg(w_{2,n}-\frac{m_n}{2|B_n|}\bigg),\bigg(u_n-\bigg(w_{1,n}-\frac{m_n}{2|A_n|} \bigg),v_n-\bigg(w_{2,n}-\frac{m_n}{2|B_n|}\bigg)\bigg)\bigg)\notag\\
    &=a_n((u_n,v_n),(u_n,v_n))-2a_n\bigg((u_n,v_n),\bigg(w_{1,n}-\frac{m_n}{2|A_n|}, w_{2,n}-\frac{m_n}{2|B_n|}\bigg)\bigg)\label{c1unvn}\\
    &\qquad +a_n\bigg(\bigg(\frac{m_n}{2|A_n|},\frac{m_n}{2|B_n|}\bigg),\bigg(\frac{m_n}{2|A_n|},\frac{m_n}{2|B_n|}\bigg)\bigg)- 2 a_n ((w_{1,n}, w_{2,n}) , \bigg(\frac{m_n}{2|A_n|},\frac{m_n}{2|B_n|}\bigg) \bigg)\label{c2unvn} \\
    &\qquad +a_n((w_{1,n},w_{2,n}),(w_{1,n},w_{2,n}))\label{stress}.
\end{align}

To deal with \eqref{c2unvn}, we observe that
\begin{align*}&a_n\bigg(\bigg(\frac{m_n}{2|A_n|},\frac{m_n}{2|B_n|}\bigg),\bigg(\frac{m_n}{2|A_n|},\frac{m_n}{2|B_n|}\bigg)\bigg)-2a_n \bigg((w_{1,n}, w_{2,n}) , \bigg(\frac{m_n}{2|A_n|},\frac{m_n}{2|B_n|}\bigg)\bigg)\\&=\displaystyle\iint_{A_n\times B_n}J(x-y)\bigg(\frac{m_n}{2|A_n|}-\frac{m_n}{2|B_n|}\bigg)^2dydx
\\& \qquad -2\iint_{A_n\times B_n} J(x-y)(w_{2,n}(y)-w_{1,n}(x))\bigg(\frac{m_n}{2|A_n|}-\frac{m_n}{2|B_n|}\bigg)dydx.
 \end{align*}
Rewriting the integrals over $\Omega$, we get
\begin{align*}
&\iint_{\Omega\times \Omega}J(x-y)\chi_{A_n}(x)\chi_{B_n}(y)dydx\bigg(\frac{m_n}{2|A_n|}-\frac{m_n}{2|B_n|}\bigg)^2\\ 
& \quad - \iint_{\Omega\times \Omega} J(x-y)\chi_{A_n}(x)\chi_{B_n}(y)(w_{2,n}(y)-w_{1,n}(x))dydx\bigg(\frac{m_n}{2|A_n|}-\frac{m_n}{2|B_n|}\bigg)+o\bigg(\frac{1}{n}\bigg)\longrightarrow 0, \qquad \text{as} \quad n\to \infty,\end{align*}
Indeed,
\begin{align}\label{mnto0}
&\frac{m_n}{2|A_n|}-\frac{m_n}{2|B_n|}=m_n\bigg(\frac{1}{2|A_n|}-\frac{1}{2|B_n|}\bigg)=\bigg(\int_{A_n}w_{1,n}dx+\int_{B_n}w_{2,n}dx\bigg)\bigg(\frac{1}{2|A_n|}-\frac{1}{2|B_n|}\bigg)\notag\\
&=\bigg(\int_{\Omega}\chi_{A_n}(x)udx+\int_{\Omega}\frac{\chi_{B_n}v}{1-X}(x)dx+o\bigg(\frac{1}{n}\bigg)\bigg)\Bigg(\frac{1}{2\displaystyle\int_{\Omega}\chi_{A_n}(x)dx}-\frac{1}{2\displaystyle\int_{\Omega}\chi_{B_n}(x)dx}\Bigg)\notag\\
&\longrightarrow \bigg(\int_{\Omega}Xudx+\int_{\Omega}v(x)dx\bigg)\Bigg(\frac{1}{2\displaystyle\int_{\Omega}Xdx}-\frac{1}{2\displaystyle\int_{\Omega}(1-X)(x)dx}\Bigg)=0, \qquad \text{as} \quad n\to \infty,
\end{align}
using the assumption $$0=\displaystyle\int_{A_n}u_n (x) dx+\int_{B_n}v_n(x) dx,$$ whose limit is given by
$$\int_{A_n}u_n (x) dx+\int_{B_n}v_n (x) dx\longrightarrow \int_{\Omega}Xu(x) dx+\int_{\Omega}v (x) dx=0,\qquad \text{as} \quad n\to \infty.$$

The term in \eqref{c1unvn} can be handle using \eqref{solution} and \eqref{mnto0}, 
\begin{align*}
    &a_n((u_n,v_n),(u_n,v_n))-2a_n\bigg((u_n,v_n),\bigg(w_{1,n}-\frac{m_n}{2|A_n|}, w_{2,n}-\frac{m_n}{2|B_n|}\bigg)\bigg)
    \\
    &=-(f_n,u_n)_{L^2(A_n)}-(f_n,v_n)_{L^2(B_n)}+2\bigg(f_n,w_{1,n}-\frac{m_n}{2|A_n|}\bigg)_{L^2(A_n)}+2\bigg(f_n,w_{2,n}-\frac{m_n}{2|B_n|}\bigg)_{L^2(B_n)}\\
    &=-\displaystyle\int_{\Omega}\chi_{A_n}u_nf_ndx-\int_{\Omega}\chi_{B_n}v_nf_ndx+2\int_{\Omega}\chi_{A_n}f_nudx+2\int_{\Omega}f_n\frac{\chi_{B_n}v}{1-X}dx\\
    &\quad \quad \quad \quad \quad \quad \quad \quad \quad  -2\int_{\Omega}\chi_{A_n}f_n\frac{m_{n}}{2|A_n|}-2\int_{\Omega}\chi_{B_n}f_n\frac{m_{n}}{2|B_n|}+o\bigg(\frac{1}{n}\bigg) \\
    &\longrightarrow -\int_{\Omega}Xufdx-\int_{\Omega}vfdx+2\int_{\Omega}fXudx+2\int_{\Omega}fvdx=\int_\Omega fXudx+\int_{\Omega}fvdx \qquad \text{as} \quad n\to \infty.
    \end{align*}

The most delicate step is to pass to the limit in the last term \eqref{stress}.  
This limit turns out to be
\[
-\left(\int_\Omega fXu + \int_\Omega fv\right),
\]
which matches, with opposite sign, the limit of the previous contribution in \eqref{c1unvn}.  
Hence both terms cancel, and the whole expression (\eqref{c1unvn}, \eqref{c2unvn} and \eqref{stress})  converges to zero. To compute the last limit we observe that
\begin{align}
& a_n((w_{1,n},w_{2,n}),(w_{1,n},w_{2,n}))=\int_{A_n} \bigg|\nabla \bigg(u(x)-\frac{1}{n}\sum_{i=1}^NU^i(nx)\frac{\partial u}{\partial x_i}(x)\bigg)\bigg|^2dx\label{1stc}\\
&\qquad +\iint_{A_n\times B_n} J(x-y)\bigg(\frac{\chi_{B_n}v}{1-X}(y)-\bigg(u(x)-\frac{1}{n}\sum_{i=1}^N \chi_{A_n}(x)U^i(nx)\frac{\partial u}{\partial x_i}(x)\bigg)\bigg)^2dydx\label{2ndc}\\
&\qquad +\frac{1}{2}\iint_{B_n \times B_n} G(x-y)\bigg(\frac{\chi_{B_n}v}{1-X}(y)-\frac{\chi_{B_n}v}{1-X}(x)\bigg)^2dydx\label{3rdc}
\end{align}
We will analyze each term separately. 

For the first term \eqref{1stc}, we expand the square and observe that terms coming from the corrector terms reduce to $o(1/n)$
\begin{align*}
  &  \int_{A_n}|\nabla u(x)|^2-2\nabla u(x)\sum_{i=1}^N\nabla_y U^i(nx)\frac{\partial u}{\partial x_i}(x)+\bigg[\sum_{i=1}^N \nabla_y U^i(nx)\frac{\partial u}{\partial x_i}(x)\bigg)\bigg]^2dx+o\bigg(\frac{1}{n}\bigg)\end{align*}
We will pass to the limit the following decomposition, using the regularity of $u$ \footnote{If $f\in L^2(\Omega)$, we get $\nabla u\in L^2(A)$ in \eqref{nlh}--\eqref{lh}}: 
 \begin{align}
&\int_{A_n}|\nabla u(x)|^2-\nabla u(x)\sum_{i=1}^N\nabla_y U^i(nx)\frac{\partial u}{\partial x_i}(x)dx+o\bigg(\frac{1}{n}\bigg) \label{l1c}\\
&+\int_{A_n}\bigg[\sum_{i=1}^N \nabla_y U^i(nx)\frac{\partial u}{\partial x_i}(x)\bigg)\bigg]^2-\nabla u(x)\sum_{i=1}^N\nabla_y U^i(nx)\frac{\partial u}{\partial x_i}(x)dx \label{l2c}
\end{align}
where we can rewrite the above expression as 
\begin{align}
&\int_{A_n} \nabla u(x) \left( I - M(n x) \right) \nabla u(x) \, dx + o\left(\frac{1}{n}\right)\label{passlimdiv}\\
&+\int_{A_n} [M(nx)\nabla u(x)-\nabla u(x)][M(nx)\nabla u(x)]dx\label{passlim0},
\end{align}
where 
\begin{equation}\label{matrix}
M(nx):=\bigg[\frac{\partial U^i}{\partial y_j}\bigg]_{i,j=1}^N=
\begin{pmatrix}
\frac{\partial U^1}{\partial y_1} & \frac{\partial U^2}{\partial y_1} & \cdots & \frac{\partial U^N}{\partial y_1} \\
\frac{\partial U^1}{\partial y_2} & \frac{\partial U^2}{\partial y_2} & \cdots & \frac{\partial U^N}{\partial y_2} \\
\vdots & \vdots & \ddots & \vdots \\
\frac{\partial U^1}{\partial y_N} & \frac{\partial U^2}{\partial y_N} & \cdots & \frac{\partial U^N}{\partial y_N}
\end{pmatrix}(nx).\end{equation}
Notice that, since we have the equation \eqref{Ui}, we can perform the following computation,
\begin{align*}
    0=-\int_{Y^*}U^l(\Delta U^i)&=\int_{Y^*}\nabla U^i(y)\nabla U^l(y)dy-\int_{\partial Y^*_{int}}U^l(y)\underbrace{\frac{\partial U^i}{\partial \eta}}_{=\eta_i}(y)dS(y)\\
    &=\int_{Y^*}\nabla U^i(y)
    \nabla U^l(y)dy-\int_{\partial Y^*_{int}}[U^le_i](y)\eta dS_i(y)\\
    &=\int_{Y^*}\nabla U^i(y)
    \nabla U^l(y)dy-\int_{ Y^*}div [U^le_i](y)dy\\
    &=\int_{Y^*}\nabla U^i(y)\nabla U^l(y)dy-\int_{Y^*}\frac{\partial U^l}{\partial y_i}dy.
\end{align*}
Then, we conclude that
\begin{equation*}
\int_{Y^*}\nabla U^i(y)\nabla U^l(y)dy=\int_{Y^*}\frac{\partial U^l}{\partial y_i}dy,\end{equation*}
which proves the following property of the matrix \eqref{matrix}:
\begin{equation}\label{sM}
\int_{Y^*}\nabla U^i(y)\nabla U^l(y)dy=\int_{Y* }\frac{\partial U^i}{\partial y_l}dy=\int_{Y^*}\frac{\partial U^l}{\partial y_i}dy.\end{equation}
Then, we can rewrite \eqref{passlimdiv}--\eqref{passlim0} as 
\begin{align}
    \int_{A_n} \nabla u(x) \left( I - M(n x) \right) \nabla u(x) \, dx &+\displaystyle\int_{A_n} [(M(nx)^2-M(nx))\nabla u]
\cdot \nabla u(x)dx+ o\left(\frac{1}{n}\right)\notag\\
   & = \int_{\Omega} \nabla u(x) \widetilde{\left( I - M(n x) \right)} \nabla u(x) \, dx + o\left(\frac{1}{n}\right)\label{limhom} \\
&\qquad +\displaystyle\int_{\Omega} [\widetilde{\mathit{(M(nx)^2-M(nx))}}\nabla u]
\cdot \nabla u(x)dx. \label{lim0}
\end{align}
Now we can pass to the limit, as $n\to \infty$. The term \eqref{limhom} can be treated as follows,
\begin{equation}\label{localcorrector}
\lim_{n\to \infty}\int_\Omega \nabla u(x) \widetilde{\bigg[I - M(nx)\bigg]} \nabla u(x) dx
=\int_\Omega \nabla u(x) \, Q_{\rm Hom} \, \nabla u(x)dx,
\end{equation}
where $Q_{\rm Hom}$ is the matrix of the homogenized coefficients given by 
$$
\left( Q_{\rm Hom} \right)_{ij} = q_{ij} = \frac{1}{|Y|}\int_{Y^*}I-\bigg[\frac{\partial U^i}{\partial y_j}\bigg]_{i,j=1}^Ndx.
$$
Notice that, by the Average Convergence for Periodic Functions [see \cite{doina2}, p. xvi], we have  
$$
\widetilde{\bigg[I - \bigg[\frac{\partial U^i}{\partial y_j}\bigg]_{i,j=1}^N(nx)\bigg]} \rightharpoonup  q_{ij} \quad \textrm{weakly$^*$ in } L^\infty(\Omega).  
$$
The last part, \eqref{lim0}, vanishes in the limit. To show this, we use the property \eqref{sM} of the matrix \eqref{matrix},
\begin{equation*}\lim_{n\to \infty}\displaystyle\int_{\Omega} [\widetilde{\mathit{(M(nx)^2-M(nx))}}\nabla u]
\cdot \nabla u(x)dx=\frac{1}{|Y|}\int_{\Omega} [\tilde{M}-\tilde{M}]\nabla u(x)\cdot \nabla u(x)dx,\end{equation*}
where $\tilde{M}$ denotes the matrix
\begin{equation*}
\tilde{M}=
\begin{pmatrix}
\int_{Y^*}\frac{\partial U^1}{\partial y_1}dx & \cdots & \int_{Y^*}\frac{\partial U^N}{\partial y_1}dx \\
\int_{Y^*}\frac{\partial U^1}{\partial y_2}dx & \cdots & \int_{Y^*}\frac{\partial U^N}{\partial y_2}dx \\
\vdots  & \ddots & \vdots \\
\int_{Y^*}\frac{\partial U^1}{\partial y_N}dx & \cdots & \int_{Y^*}\frac{\partial U^N}{\partial y_N}dx
\end{pmatrix}.\end{equation*}

Now we deal with \eqref{2ndc}. We can write the square term as follows:
\begin{align*}
    &\bigg(\frac{\chi_{B_n}v}{1-X}(y)-\bigg(u(x)-\frac{1}{n}\sum_{i=1}^m \chi_{A_n}(x)U^i(nx)\frac{\partial u}{\partial x}(x)\bigg)\bigg)^2\\
    &= \bigg(\frac{\chi_{B_n}v^2}{(1-X)^2}(y)-2 \bigg(\frac{\chi_{B_n}v}{1-X}(y)\bigg)\bigg(u(x)-\frac{1}{n}\sum_{i=1}^m \chi_{A_n}(x)U^i(nx)\frac{\partial u}{\partial x}(x)\bigg)
    \\ & \quad +\bigg(u(x)-\frac{1}{n}\sum_{i=1}^m \chi_{A_n}(x)U^i(nx)\frac{\partial u}{\partial x}(x)\bigg)^2\\
    &=u^2(x)-2u(x)\frac{\chi_{B_n}v(y)}{1-X}+\frac{\chi_{B_n}(y)v^2(y)}{(1-X)^2}+o\bigg(\frac{1}{n}\bigg),
\end{align*}
Then, using $\chi_{B_n}^2=\chi_{B_n}$, we obtain
\begin{align*}
 &\iint_{A_n\times B_n} J(x-y)\bigg(u^2(x)-2u(x)\frac{\chi_{B_n}v(y)}{1-X}+\frac{\chi_{B_n}(y)v^2(y)}{(1-X)^2}+o\bigg(\frac{1}{n}\bigg)\bigg)dydx\\
 &=\iint_{\Omega \times \Omega} J(x-y)\bigg(\chi_{A_n}(x)\chi_{B_n}(y)u^2(x)-2\chi_{A_n}(x)u(x)\frac{\chi_{B_n}v(y)}{1-X}+\chi_{A_n}(x)\frac{\chi_{B_n}(y)v^2(y)}{(1-X)^2}+o\bigg(\frac{1}{n}\bigg)\bigg)dydx\\
 &=\iint_{\Omega \times \Omega} J(x-y)\bigg(\chi_{A_n}(x)\chi_{B_n}(y)u^2(x)-\chi_{A_n}(x)u(x)\frac{\chi_{B_n}v(y)}{1-X}\bigg)dydx\\
 & \qquad +\iint_{\Omega\times \Omega} J(x-y)\bigg(\chi_{A_n}(x)\frac{\chi_{B_n}(y)v^2(y)}{(1-X)^2}-\chi_{A_n}(x)u(x)\frac{\chi_{B_n}v(y)}{1-X}\bigg)dydx+o\bigg(\frac{1}{n}\bigg).
\end{align*}
Notice that we can pass to the limit as $n\to \infty$, obtaining
\begin{align*}
&\iint_{\Omega \times \Omega} J(x-y)\bigg(X(1-X)u^2(x)-Xu(x)\frac{(1-X)v(y)}{1-X}\bigg)dydx\\
 & \qquad +\iint_{\Omega \times \Omega} J(x-y)\bigg(X\frac{1-Xv^2(y)}{(1-X)^2}-Xu(x)\frac{(1-X)v(y)}{1-X}\bigg)dydx
\end{align*}
that we can rewrite as follows:
\begin{align}\label{Jcorrectorp}
\iint_{\Omega \times \Omega} J(x-y)\,X\bigg((1-X)u(x)&-v(y)\bigg)u(x)dydx\\
& \qquad + \iint_{\Omega \times \Omega} J(x-y)\frac{v(y)}{1-X}\,X\bigg(v(y)-(1-X)u(x)\bigg)dydx\notag
\end{align}

To handle the last term, \eqref{3rdc}, we argue as follows. First, we observe that, since $G$ is symmetric, we can rewrite the nonlocal term as
\begin{align*}
&-\int_{B_n\times B_n} G(x-y)\bigg(\frac{\chi_{B_n}v}{1-X}(y)-\frac{\chi_{B_n}v}{1-X}(x)\bigg)\bigg(\frac{\chi_{B_n}v}{1-X}(x)\bigg)dydx\\
&=\int_{B_n}\int_{B_n}G(x-y)dy\frac{\chi_{B_n}v}{1-X}(x)\frac{\chi_{B_n}v}{1-X}(x)dx-\int_{B_n}\int_{B_n} G(x-y)\frac{\chi_{B_n}v}{1-X}(y)dy\frac{\chi_{B_n}v}{1-X}(x)dy\\
&=\int_{\Omega}\int_{\Omega}G(x-y)\chi_{B_n}(y)dy\frac{\chi_{B_n}v^2}{(1-X)^2}(x)dx-\int_{\Omega}\int_{\Omega}G(x-y)\frac{\chi_{B_n}v}{1-X}(y)dy\frac{\chi_{B_n}v}{1-X}(x)dy.
\end{align*}
Since $\chi_{B_n}\xrightharpoonup{*}1-X \in L^{\infty}(\Omega)$ and the other $L^2(\Omega)$-functions are fixed, when $n\to \infty$ we get
\begin{align}\label{Gcorrectorp}
    -\iint_{\Omega\times \Omega} G(x-y)&\frac{v(x)}{1-X}((1-X)v(y)-(1-X)v(x))dydx\\
    & = \frac12\iint_{\Omega\times\Omega} G(x-y)((1-X)v(y)-(1-X)v(x))\bigg(\frac{v(y)}{1-X}-\frac{v(x)}{1-X}\bigg)dydx.
\end{align}

Notice that, by \eqref{localcorrector},\eqref{Jcorrectorp} and \eqref{Gcorrectorp}, in the limit we get the bilinear form associated to the system \eqref{nlh}--\eqref{lh},
\begin{align*}
a_{\infty}((u,v),(u,v))&=\int_\Omega Q_{\rm Hom}\nabla u(x)\cdot \nabla u(x)dx+\int_{\Omega}u(x)\int_{\Omega} J(x-y)\,X\bigg((1-X)u(x)-v(y)\bigg)dydx\\&
\qquad+\int_{\Omega}\frac{v(y)}{1-X}\int_{\Omega} J(x-y)\, X\bigg(v(y)-(1-X)u(x)\bigg)dydx\\
&\qquad +\frac12\int_{\Omega}\frac{v(x)}{1-X}\int_{\Omega} G(x-y)((1-X)v(y)-(1-X)v(x))\bigg(\frac{v(y)}{1-X}-\frac{v(x)}{1-X}\bigg)dydx.\end{align*}
In fact, we obtained
$$a_n((w_{1,n},w_{2,n}),(w_{1,n},w_{2,n}))\longrightarrow a_{\infty}((u,v),(u,v))=-\int_\Omega fXudx-\int_{\Omega}fvdx, \qquad \text{as} \quad n\to \infty,$$
concluding the proof of \eqref{correctorconvergence}
\end{proof}

Finally, we prove Theorem \ref{thrmnlinballs}, where we pass to the limit of the weak formulation of \eqref{hlocal}-\eqref{hnl} to obtain the homogenized equation where $(u,v)\in H^1(\Omega)\times L^2(\Omega)$ is a weak solution of 
\eqref{nlh}-\eqref{lh}, using Proposition \ref{correctornlinballs}.  

Since $P_n u_n$ denotes the extension of $u_n$ to $\Omega$ we get
$$\bigg\|u_n - \bigg(w_{n,1}-\frac{m_n}{2|A_n|}\bigg)\bigg\|_{H^1(A_n)} \ge C\,\bigg\|P_n \bigg(u_n - \bigg(w_{n,1}-\frac{m_n}{2|A_n|}\bigg)\bigg)\bigg\|_{H^1(\Omega)}$$
for some constant $C>0$ independent of $n$.

\begin{proof}[Proof of Theorem \ref{thrmnlinballs}] Let $(\varphi,\phi)\in H^1(\Omega)\times L^2(\Omega)$ such that 
$$X\int_\Omega \varphi(x)dx+\int_\Omega \phi(x)dx=0.$$
Define $$m_n=\displaystyle\int_{A_n}\varphi(x)dx+\int_{B_n}\frac{\phi}{1-X}(x)dx,$$ and consider the vector
$$v=\bigg(\varphi-\frac{1}{2}\frac{m_n}{|A_n|
},\frac{\phi}{1-X}-\frac{1}{2}\frac{m_n}{|B_n|}\bigg)\in W_n,$$
since
$$\int_{A_n}\varphi(x)dx-\frac{1}{2}\int_{A_n}\frac{m_n}{|A_n|}dx+\int_{B_n}\frac{\phi(x)}{1-X}dx-\frac{1}{2}\int_{B_n}\frac{m_n}{|B_n|}dx=0.$$
Applying $v\in W_n$ in  \eqref{solution}, we obtain
\begin{align*}
&\int_{A_n} \nabla u_n \cdot \nabla\varphi dx+\iint_{A_n\times B_n}J(x-y)(v_n(y)-u_n(x))\bigg(\bigg(\frac{\phi(y)}{1-X}-\varphi(x)\bigg)-\frac12\bigg(\frac{m_n}{|B_n|}-\frac{m_n}{|A_n|}\bigg)\bigg)dydx\\
&\qquad +\frac12\iint_{B_n\times B_n}G(x-y)(v_n(y)-v_n(x))\bigg(\frac{\phi(y)}{1-X}-\frac{\phi(x)}{1-X}\bigg)dydx\\
&=-\int_{A_n}f_n\bigg(\varphi(x)-\frac12 \frac{m_n}{|A_n|}\bigg)dx-\int_{B_n}f_n\bigg(\frac{\phi(x)}{1-X}-\frac12 \frac{m_n}{|B_n|}\bigg)dx.
    \end{align*}
Introducing $\pm w_{n,1}$, the corrector term \eqref{w1nnl}, in this expression, we obtain
\begin{align*}
&\int_{A_n} \nabla (u_n-w_{n,1}) \nabla \varphi dx+
\int_{A_n} \nabla \bigg[u-\frac{1}{n}\sum_{i=1}^{N}\nabla_yU^i(nx)\frac{\partial u}{\partial x_i}\bigg] \nabla \varphi dx\\
&\qquad +\iint_{A_n\times B_n}J(x-y)(v_n(y)-u_n(x))\bigg(\bigg(\frac{\phi(y)}{1-X}-\varphi(x)\bigg)-\frac12\bigg(\frac{m_n}{|B_n|}-\frac{m_n}{|A_n|}\bigg)\bigg)dydx\\
&\qquad +\frac12\iint_{B_n\times B_n}G(x-y)(v_n(y)-v_n(x))\bigg(\frac{\phi(y)}{1-X}-\frac{\phi(x)}{1-X}\bigg)dydx\\
&=-\int_{A_n}f_n\bigg(\varphi(x)-\frac12 \frac{m_n}{|A_n|}\bigg)dx-\int_{B_n}f_n\bigg(\frac{\phi(y)}{1-X}-\frac12 \frac{m_n}{|B_n|}\bigg)dx.
\end{align*}
Rewriting the above equation with the characteristic functions and the extension operator $P_n$ and using the regularity of $u$, we get,
\begin{align*}
&\int_{\Omega} \chi_{A_n}(x)\nabla P_n(u_n-{w_{n,1}})\,\nabla \varphi(x) dx+
\int_{\Omega} \widetilde{\bigg(I-\bigg[\frac{\partial U^i}{\partial y_j}\bigg]_{j,i=1}^N}\bigg) \nabla u \nabla \varphi dx+o\bigg(\frac{1}{n}\bigg)\\
& \qquad +\iint_{A_n\times B_n}J(x-y)(v_n(y)-u_n(x))\bigg(\bigg(\frac{\phi(y)}{1-X}-\varphi(x)\bigg)-\frac12\bigg(\frac{m_n}{|B_n}-\frac{m_n}{|A_n|}\bigg)\bigg)dydx\\
&\qquad +\frac12\iint_{B_n\times B_n}G(x-y)(v_n(y)-v_n(x))\bigg(\frac{\phi(y)}{1-X}-\frac{\phi(x)}{1-X}\bigg)dydx\\
&=-\int_{A_n}f_n\bigg(\varphi(x)-\frac12 \frac{m_n}{|A_n|}\bigg)dx-\int_{B_n}f_n\bigg(\frac{\phi(x)}{1-X}-\frac12 \frac{m_n}{|B_n|}\bigg)dx.\end{align*}

Now, we are able to pass to the limit using Theorem \ref{correctornlinballs}, the definition of $(Q_{\rm Hom})$ and the extension estimate
\eqref{extension}. First, notice that
$$m_n=\int_\Omega \chi_{A_n}\varphi(x)dx+\int_\Omega \chi_{B_n}\frac{\phi(x)}{1-X}dx\longrightarrow \int_\Omega X\varphi(x)dx+\int_\Omega \phi(x)dx =0, \ \text{as}\ n\to \infty.$$
Then, we obtain
\begin{align*}\label{locallimhp}
&\int_{\Omega}Q_{\rm Hom}\nabla u(x)  \nabla \varphi(x)dx+\int_{\Omega}\int_{\Omega}J(x-y)\ X(v(y)-(1-X)u(x))\bigg(\frac{\phi(y)}{1-X}-\varphi(x)\bigg)dydx\\
&\quad -\int_\Omega\bigg[\int_\Omega G(x-y)((1-X)v(y)-(1-X)v(x))dy\bigg]\frac{\phi}{1-X}(x)dx=-\int_\Omega f(x)\varphi(x)Xdx-\int_\Omega f(x)\phi(x)dx,
\end{align*}
which we can rewrite as
\begin{align*}
&\int_{\Omega}Q_{\rm Hom}\nabla u(x)  \nabla \varphi(x)dx+\int_{\Omega}\int_{\Omega}J(x-y)\ X(v(y)-(1-X)u(x))\bigg(\frac{\phi(y)}{1-X}-\varphi(x)\bigg)dydx\\
&+\frac12\int_\Omega\bigg[\int_\Omega G(x-y)((1-X)v(y)-(1-X)v(x))\bigg(\frac{\phi}{1-X}(y)-\frac{\phi}{1-X}(x)\bigg)dydx\\
&\qquad \qquad \qquad \qquad \qquad \qquad \qquad \qquad =-\int_\Omega f(x)\varphi(x)Xdx-\int_\Omega f(x)\phi(x)dx,
\end{align*}
for all $\varphi,\phi \in H^1(\Omega)\times L^2(\Omega)$, obtaining the weak formulation of \eqref{lh}--\eqref{nlh} with Neumann boundary conditions.

Finally, we prove that the solutions to the homogenized system 
\eqref{nlh}--\eqref{lh} are unique. Let $(u,v),(\tilde{u},\tilde{v})\in H^1(\Omega)\times L^2(\Omega)$ 
be two solutions of the system satisfying, respectively,
\[
\int_{\Omega}\!\big(Xu(x)+v(x)\big)\,dx = 0,
\qquad
\int_{\Omega}\!\big(X\tilde{u}(x)+\tilde{v}(x)\big)\,dx = 0.
\]
Consider the differences
\[
w_u := u - \tilde{u}, 
\qquad 
w_v := v - \tilde{v}.
\]
Subtracting the two systems \eqref{lh}--\eqref{nlh} yields
\begin{align}\label{ulnlinball}
0
&= \operatorname{div}\!\big(Q_{\rm Hom}\nabla w_u(x)\big)
  + \int_{\Omega} J(x-y)\,X
    (w_v(y)-(1-X)w_u(x))\,dy,
    \qquad x\in\Omega,\\[2mm]
\label{unlnlinball}
0
&= \int_{\Omega} G(x-y)
    \big[(1-X)w_v(y) - (1-X)w_v(x)\big]\,dy \notag\\
&\quad + \int_{\Omega} J(x-y)\,X
    ((1-X)w_u(y) - w_v(x))\,dy,
    \qquad x\in\Omega.
\end{align}
Multiplying \eqref{ulnlinball} by $w_u \in H^1(\Omega)$ 
and \eqref{unlnlinball} by $\dfrac{w_v}{1-X} \in L^2(\Omega)$, 
and integrating, we obtain
\begin{align}\label{unicJ}
0
&= \int_{\Omega} 
    \operatorname{div}\!\big(Q_{\rm Hom}\nabla w_u(x)\big)\,w_u(x)\,dx
  + \iint_{\Omega\times\Omega}
    J(x-y)\,X\,(w_v(y) -(1-X) w_u(x))\,w_u(x)\,dy\,dx,
\end{align}
and
\begin{align}\label{unicG}
0
&= \iint_{\Omega\times\Omega}
   G(x-y)\,(1-X)(1-X)\,
   \bigg[\frac{w_v(y)}{1-X} - \frac{w_v(x)}{1-X}\bigg]
   \frac{w_v(x)}{1-X}\,dy\,dx \notag\\
&\quad + \iint_{\Omega\times\Omega}
   J(x-y)\,X\,
   ((1-X)w_u(y) - w_v(x))\,
   \frac{w_v(x)}{1-X}\,dy\,dx.
\end{align}
Applying Green’s identity for the local term and using the symmetric hypothesis on $G$, 
we deduce
\begin{align*}
0
&= -\int_{\Omega} Q_{\rm Hom}\nabla w_u(x)\!\cdot\!\nabla w_u(x)\,dx
  + \iint_{\Omega\times\Omega}
    J(x-y)\,X(1-X)\,
    \bigg[\frac{w_v(y)w_u(x)}{1-X}-w_u^2(x)
          \bigg]\,dy\,dx,\\[2mm]
0
&= -\frac12\iint_{\Omega\times\Omega}
     G(x-y)\,(1-X)\,
     (w_v(x) -w_v(y))^2\,dy\,dx\\
&\quad + \iint_{\Omega\times\Omega}
     J(x-y)\,X(1-X)\,
     \bigg[\frac{w_v(x)w_u(y)}{1-X} 
           - \frac{w_v^2(x)}{(1-X)^2}\bigg]\,dy\,dx.
\end{align*}
Adding the two equations above, we obtain
\begin{align}\label{midterm}
0
&= -\int_{\Omega} Q_{\rm Hom}|\nabla w_u(x)|^2\,dx
   - \iint_{\Omega\times\Omega}
     J(x-y)\,X(1-X)\,
     \bigg[w_u(x) - \frac{w_v(y)}{1-X}\bigg]^2\,dy\,dx\notag\\
&\quad - \frac12\iint_{\Omega\times\Omega}
     G(x-y)\,(1-X)(1-X)\,
     \bigg[\frac{w_v(x)}{1-X} - \frac{w_v(y)}{1-X}\bigg]^2\,dy\,dx.
\end{align}
Since the kernels $J$ and $G$ are nonnegative and 
$Q_{\rm Hom}$ is positive definite (see Remark~\ref{Q_hom}), 
and $X\in(0,1)$, 
each term above is negative. 
Hence, equality can hold only if
\[
w_u(x) = c_1,
\qquad 
\frac{w_v(x)}{1-X} = c_2,
\qquad c_1,c_2\in\mathbb{R},
\]
and, from the second integral of \eqref{midterm} we also have
\[
w_u(x) - \frac{w_v(x)}{1-X} = 0,
\]
we conclude that $c_1=c_2=:c$. Finally, using the zero-mean condition
\[
\int_{\Omega} Xw_u(x)\,dx 
+ \int_{\Omega} w_v(x)\,dx = 0,
\]
we obtain
\[
0 = c\!\int_{\Omega} \big(X + 1 - X\big)\,dx = c|\Omega|,
\]
which implies $c=0$. 
Thus $w_u\equiv 0$ and $w_v\equiv 0$, 
and consequently $u=\tilde{u}$ and $v=\tilde{v}$. 
Therefore, the homogenized system \eqref{lh}--\eqref{nlh} 
admits a unique solution.
\end{proof}

\section{Strip Periodic Case} \setcounter{equation}{0}

In this section, we analyze the homogenization of the mixed local–nonlocal problem \eqref{hlocal}–\eqref{hnl} in the particular case where the spatial domain is the unit cube $\Omega = [0,1]^N \subset \mathbb{R}^N$, $N\ge 2$. Our goal is to understand the effective behavior of the system when $\Omega$ is decomposed into a partition of horizontal strip. More precisely, for each $n\ge 2$, we construct a family of subsets $(A_n, B_n)$ by dividing $\Omega$ into $2^{n-1}$ horizontal strips of equal height and selecting the union of every even strip to be $A_n$ (recall \eqref{strip}), with $B_n$ being its complement.

Our main result in this section is stated in Theorem~\ref{thrmstrip}. Before presenting its proof, we first introduce the homogenized bilinear form that emerges as the effective limit of the oscillatory sequence of mixed problems described above.

The pair $(u,v)\in  L^2([0,1]_{x_N}; H^1([0,1]^{N-1}))\times L^2([0,1]^N)$ with the property 
$$\displaystyle\int_{[0,1]^N} u(x)dx+\int_{[0,1]^N} v(x)dx=0$$
is called the weak solution to the homogenized equation \eqref{hstripslocal}--\eqref{hstripsnl} if the integral identity
\begin{equation}\label{lmlimit}
a_\infty((u,v),(\varphi,\phi))=-F_\infty(\varphi,\phi)\end{equation}
holds for all $(\varphi,\phi)\in L^2([0,1]_{x_N}; H^1([0,1]^{N-1}))\times L^2(0,1)$ with $\displaystyle\int_{[0,1]^N} \varphi(x)dx+\int_{[0,1]^N} \phi(x)dx=0$, where
\begin{align*}
    a_\infty((u,v),(\varphi,\phi))&=\sum_{i=1}^{N-1}\int_{[0,1]^N}\frac{\partial u}{\partial x_i} \frac{\partial \varphi}{\partial x_i}(x)dx\\
    &\qquad +\iint_{[0,1]^N\times [0,1]^N} J(x-y)(Xv(y)-(1-X)u(x))(\phi(y)-\varphi(x))dxdy\\&
    \qquad +\frac{1}{2}\iint_{[0,1]^N\times [0,1]^N} G(x-y)((1-X)v(y)-(1-X)v(x))(\phi(y)-\phi(x))dxdy\\
\end{align*}
and 
\begin{align*}
F_\infty(\varphi,\phi)=\int_{[0,1]^N} X\varphi(x)f(x)dx+\int_{[0,1]^N} (1-X)\phi(x)f(x)dx.
\end{align*}

We are now ready to prove Theorem \ref{thrmstrip}.

\begin{proof}[Proof of Theorem \ref{thrmstrip}]
Recall from \eqref{solution} that the bilinear form is written as
\begin{align}\label{hlocalstrip}
    &-\int_{A_n} \varphi(x)f_n(x)dx-\int_{B_n} \phi(x)f_n(x)dx=\int_{A_n}\nabla  u_n(x)\nabla \varphi(x)dx\\
    &\qquad +\iint_{A_n\times B_n} J(x-y)(v_n(y)-u_n(x))(\phi(y)-\varphi(x))dydx \notag 
    \\
    &\qquad  +\frac{1}{2}\iint_{B_n\times B_n} G(x-y)(v_n(y)-v_n(x))(\phi(y)-\phi(x))dydx,\notag
\end{align}
for all $(\varphi,\phi)\in H^1(A_n)\times L^2(B_n),$ with $\displaystyle\int_{A_n} \varphi(x)dx+\int_{B_n} \phi(x)dx=0$. 
Due to the geometric structure of the strips, the characteristic
function of $A_n,B_n$ depends only on the vertical coordinate $x_N$. In other words, we can denote
\begin{align}\label{chian}
\chi_{A_n}(x_N) := \chi_{A_n}(\hat{x},x_N), \qquad\text{and}\qquad
\chi_{B_n}(x_N) := \chi_{B_n}(\hat{x},x_N).
\end{align}
This fact allows us to treat
$\chi_{A_n}, \, \chi_{B_n}$ as a one-dimensional function in $x_N$ when we pass to the limit in the
homogenization process. 

Then, rewriting \eqref{hlocalstrip} in terms of the characteristics functions and manipulating the local term with the definition of $A_n$ we obtain
\begin{align}
&  -\int_{[0,1]^N} \chi_{A_n}(x_N)\varphi(x)f_n(x)d\hat{x}dx_N-\int_{[0,1]^N} \chi_{B_n}(x_N)\phi(x)f_n(x)d\hat{x}dx_N \notag\\
&=\sum_{i=1}^{N-1}\int_{A_n}\frac{\partial u_n}{\partial x_i}\frac{\partial \varphi}{\partial x_i}(x)dx+\sum_{\substack{k=1 \\ k \text{ even}}}^{2^{n-1}}\int_{[0,1]^{N-1}}\!\!\int_{\frac{k-1}{2^{n-1}}}^{\frac{k}{2^{n-1}}}
\frac{\partial u_n}{\partial x_N}(\hat{x},x_N)\,
\frac{\partial \varphi}{\partial x_N}(\hat{x},x_N)
\, dx_N\, d\hat{x}\label{localstrip}\\
    &\qquad +\iint_{[0,1]^N\times [0,1]^N} J(x-y)\chi_{B_N}(y_N)\chi_{A_n}(x_N)(v_n(y)-u_n(x))(\phi(y)-\varphi(x))dyd\hat{x}dx_N\notag \\
    &\qquad +\frac{1}{2}\iint_{[0,1]^N\times [0,1]^N} G(x-y)\chi_{B_n}(y_N)\chi_{B_n}(x_N)(v_n(y)-v_n(x))(\phi(y)-\phi(x))dyd\hat{x}dx_N,\quad \forall (\varphi,\phi)\in W_n.\label{hlocalstripp}
\end{align}

Our main tool for dealing with the strip structure is the construction of a special test function $\{\tilde{\varphi}_n:[0,1]^N\to \mathbb{R}\}$. To this end, consider $\xi(\hat{x})\in C^\infty([0,1]^{N-1})$, $\psi(x_N)\in C^\infty([0,1])$ and $\phi\in L^2([0,1]^N)$ and define 
\begin{equation}
\tilde{\varphi}(\hat{x},x_N)
=
\begin{cases}
    \phi(\hat{x},x_N), 
    & \text{if } (\hat{x},x_N)\in B_n, \\[8pt]
    \varphi:=\xi(\hat{x})\,\psi(x_N), 
    & \text{if } (\hat{x},x_N)\in 
   A_n := \displaystyle\bigcup_{\substack{k=1 \\ k \text{ even}}}^{2^{n-1}} (S_n)_k.
\end{cases}
\end{equation}
Assume the condition
\[
\int_{A_n} \tilde{\varphi}(x)\,dx+\int_{B_n} \phi(x)dx = 0.
\]
and the boundary conditions
\begin{equation}\label{conditionboundary}
\frac{\partial \tilde{\varphi}}{\partial \eta} = 0 \ \text{on } \partial A_n,
\qquad
\frac{\partial \tilde{\varphi}}{\partial x_i}(\hat{x},x_N)\Big|_{x_N=0}^{x_N=1} = 0,
\quad \text{for all } i = 1,\ldots, N-1.
\end{equation}

To prepare the passage to the limit, we define $\{\tilde{\varphi}_n:[0,1]^N\to \mathbb{R}\}$, replacing in $\tilde\varphi$ the dependence in $x_N$ by the average over each strip for each $k$, given by
\begin{equation}\label{specialtestfunctionstrip}
\tilde{\varphi}_n(\hat{x},x_N)
=
\begin{cases}
    \phi(\hat{x},x_N), 
    & \text{if } (\hat{x},x_N)\in B_n, \\[8pt]
    \varphi_n:=\displaystyle 
    \xi(\hat{x})\,
    \fint_{\frac{k-1}{2^{n-1}}}^{\frac{k}{2^{n-1}}}
    \psi(z)\,dz, 
    & \text{if } (\hat{x},x_N)\in 
   A_n := \displaystyle\bigcup_{\substack{k=1 \\ k \text{ even}}}^{2^{n-1}} (S_n)_k .
\end{cases}
\end{equation}
This sequence of special test functions satisfy two uniform convergence properties. Both convergences are proved by arguments analogous to those used in the case of \eqref{especialtestfunction}. The first one is the uniformly convergence to $\tilde\varphi$,
\begin{equation}\label{specialfunctionstripcv}
\tilde{\varphi}_n \longrightarrow \tilde{\varphi} 
\qquad \text{as } n\to\infty,
\qquad \text{uniformly in } [0,1]^N.
\end{equation}
The second uniform convergence is on the second derivatives. It follows from the fact that, inside each
strip of $A_n$, $\tilde{\varphi}_n$ replaces $\psi(x_N)$ in the definition of $\tilde\varphi$ by its average over the strip,
which converges uniformly to $\psi(x_N)$ due to the continuity of $\psi$, while all derivatives in the $\hat{x}$--variables depend only on the fixed smooth function $\xi$, then for each $i=1,\ldots, N-1$, we have
\begin{equation}\label{convergencespecialfunction}
\frac{\partial^2 \tilde{\varphi}_n}{\partial x_i^2}
\longrightarrow 
\frac{\partial^2 \tilde{\varphi}}{\partial x_i^2}
\qquad \text{as } n\to\infty,
\qquad \text{uniformly in } [0,1]^N.
\end{equation}
Another key property of the construction of $\tilde{\varphi}_n$ is that its partial 
derivative with respect to $x_N$ vanishes inside $A_n$, since the average of $\psi$ is constant (but depends on $k$) and $\xi$ does not depend on $x_N$.
We have that
\begin{equation}\label{N-derivativevanishes}
\frac{\partial \tilde{\varphi}_n}{\partial x_N}(\hat{x},x_N) = 0, 
\qquad (\hat{x},x_N)\in A_n:= \displaystyle\bigcup_{\substack{k=1 \\ k \text{ even}}}^{2^{n-1}} (S_n)_k.
\end{equation}
Consequently, for each fixed $n$ we obtain
\begin{equation}\label{specialfunctionvanish}
\int_{[0,1]^{N-1}}\!\!\int_{\frac{k-1}{2^{n-1}}}^{\frac{k}{2^{n-1}}}
\frac{\partial u_n}{\partial x_N}(\hat{x},x_N)\,
\frac{\partial \tilde{\varphi}_n}{\partial x_N}(\hat{x},x_N)\,
dx_N\,d\hat{x}=0.
\end{equation}

\hide{We also get the $L^2$-convergence of $\frac{\partial \tilde{\varphi}_n}{\partial x_i}\to \frac{\partial \tilde{\varphi}}{\partial x_i}$ for each $i=1,\ldots,N-1$. Indeed:

\begin{align*}
\int_\Omega \bigg|\frac{\partial \tilde{\varphi}_n}{\partial x_i}- \frac{\partial \tilde{\varphi}}{\partial x_i}\bigg|^2dx&=\underbrace{\int_{B_n} \bigg|\frac{\partial \tilde{\varphi}_n}{\partial x_i}- \frac{\partial \varphi}{\partial x_i}\bigg|^2dx}_{=0}+\int_{A_n} \bigg|\frac{\partial \varphi_n}{\partial x_i}- \frac{\partial \varphi}{\partial x_i}\bigg|^2dx\\
&=\sum_{k=1}^{2^{n-1}}\int_{[0,1]^{N-1}}\int_{\frac{k-1}{2^{n-1}}}^{\frac{k}{2^{n-1}}}\bigg|\displaystyle\fint_{\frac{k-1}{2^{n-1}}}^{\frac{k}{2^{n-1}}}\frac{\partial \varphi_n}{\partial x_i}(x,\tilde{y})d\tilde{y}- \bigg\displaystyle\fint_{\frac{k-1}{2^{n-1}}}^{\frac{k}{2^{n-1}}}\frac{\partial \varphi}{\partial x_i}(x)d\tilde{y}\bigg|^2dx_Ndx\\
&\leq\int_{[0,1]^{N-1}}\sum_{k=1}^{2^{n-1}}\int_{\frac{k-1}{2^{n-1}}}^{\frac{k}{2^{n-1}}}\fint_{\frac{k-1}{2^{n-1}}}^{\frac{k}{2^{n-1}}} \bigg|\frac{\partial \varphi_n}{\partial x_i}(x,\tilde{y})- \frac{\partial \varphi}{\partial x_i}(x)\bigg|^2d\tilde{y}dx_Ndx\\
&\leq  \frac{2^{2(n-1)}}{2^{2(n-1)}}\epsilon^2,
\end{align*}
where we use the fact that $\frac{\partial\varphi}{\partial x_i}$ is smooth, it is uniformly continuous. That is, for any $\epsilon>0$ and $i=1,\ldots, N-1$ we get $$\bigg|\frac{\partial \varphi_n}{\partial x_i}(x,\tilde{y})- \frac{\partial \varphi}{\partial x_i}(x)\bigg|<\epsilon,$$
for $|x_N-\tilde{y}|<\delta$.} 
Finally, notice that $\tilde{\varphi}_n\in W_n$, since 
\begin{align*}\int_{A_n}\varphi_n (x) \, dx+\int_{B_n}\phi (x) \, dx=&\sum_{\substack{k=1 \\ k \text{ even}}}^{2^{n-1}}\int_{[0,1]^{N-1}}\xi(\hat x)\,d\hat{x}\int_{\frac{k-1}{2^{n-1}}}^{\frac{k}{2^{n-1}}}dx_N\fint_{\frac{k-1}{2^{n-1}}}^{\frac{k}{2^{n-1}}}\varphi(z)dz\, +\int_{B_n}\phi\, dx\\&=\sum_{\substack{k=1 \\ k \text{ even}}}^{2^{n-1}}\int_{[0,1]^{N-1}}\xi(\hat x)\fint_{\frac{k-1}{2^{n-1}}}^{\frac{k}{2^{n-1}}}\varphi(z)dz\,d\hat x+\int_{B_n}\phi\,dx\\
&=\int_{A_n}\varphi (x)\, dx+\int_{B_n}\phi (x)\,dx=0.
\end{align*}

Then, we substitute this test function into the bilinear form \eqref{hlocalstripp} together with \eqref{specialfunctionvanish} in the second term of the line \eqref{localstrip}, to obtain 
\begin{align}
    -\int_{[0,1]^N} &\varphi_n(x)\chi_{A_n}(x_N)f_n(x)dx-\int_{[0,1]^N} \phi(x)\chi_{B_n}(x_N)f_n(x)dx=\sum_{i=1}^{N-1}\int_{A_n}\frac{\partial u_n}{\partial x_i}\frac{\partial \varphi_n}{\partial x_i}(x)dx \\
    &+\iint_{[0,1]^N\times [0,1]^N} J(x-y)\chi_{B_n}(y_N)\chi_{A_n}(x_N)(v_n(y)-u_n(x))\phi(y)dydx\notag\\
    &-\iint_{[0,1]^N\times [0,1]^N} J(x-y)\chi_{B_n}(y_N)\chi_{A_n}(x_N)(v_n(y)-u_n(x))\varphi_n(x)dydx \notag\\
    &-\iint_{[0,1]^N\times [0,1]^N} G(x-y)\chi_{B_n}(y_N)\chi_{B_n}(x_N)(v_n(y)-v_n(x))\phi(x)dydx.\notag
\end{align}

We now analyze the first term on the right hand side of \eqref{localstrip} for each $i \neq N$. Using the divergence theorem, we have
\begin{align*}
\int_{A_n} \frac{\partial u_n}{\partial x_i} \, \frac{\partial \varphi_n}{\partial x_i} \, dx
&= \sum_{\substack{k=1 \\ k \text{ even}}}^{2^{n-1}} \int_{(S_n)_k} \frac{\partial u_n}{\partial x_i} \, \frac{\partial \varphi_n}{\partial x_i} \, dx
=\sum_{\substack{k=1 \\ k \text{ even}}}^{2^{n-1}} \Bigg( \int_{(S_n)_k} \operatorname{div}(u_n \frac{\partial\varphi_n}{\partial x_i} \mathbf{e}_i) \, dx
- \int_{(S_n)_k} u_n \, \frac{\partial^2 \varphi_n}{\partial x_i^2} \, dx \Bigg) \\
&= \sum_{\substack{k=1 \\ k \text{ even}}}^{2^{n-1}} \Bigg( - \int_{(S_n)_k} u_n \, \frac{\partial^2 \varphi_n}{\partial x_i^2} \, dx
+ \int_{\partial (S_n)_k} \operatorname{Tr}(u_n) \, \nabla \varphi_n \cdot \mathbf{e}_i \, \nu \, dS \Bigg).
\end{align*}
where $Tr$ denotes the trace operator of $u_n$. This boundary integral vanishes in our domain setting $A_n$ defined in \eqref{strip}. In fact, we
 can write, for each $i=1,\ldots, N-1$ and $k=1,\ldots, 2^{n-1}$ even,
\begin{align*}
\int_{\partial (S_n)_k} \operatorname{Tr}(u_n) \, \nabla \varphi_n \cdot \mathbf{e}_i \, \nu \, dS
&= \int_{\frac{k-1}{2^{n-1}}}^{\frac{k}{2^{n-1}}} \operatorname{Tr}(u_n) \, \nabla \varphi_n \cdot \mathbf{e}_i \, \nu \, dS \\
&\quad + \int_{(\partial (S_n)_k)_{\text{top}}} \operatorname{Tr}(u_n) \, \nabla \varphi_n \cdot \mathbf{e}_i \, \nu \, dS
+ \int_{(\partial (S_n)_k)_{\text{bot}}} \operatorname{Tr}(u_n) \, \nabla \varphi_n \cdot \mathbf{e}_i \, \nu \, dS.
\end{align*}
Now, we observe that  
\begin{itemize}
    \item On the top horizontal interface $(\partial (S_n)_k)_{\text{top}}$ (where $x_N = \text{const}$), the normal satisfies $\nu = e_N$, so for $i \neq N$; hence the integral is zero.  
    \item On the bottom horizontal interface $(\partial (S_n)_k)_{\text{bot}}$ (where $x_N = \text{const}$), the normal satisfies $\nu = -e_N$; so again for $i \neq N$, the integral is zero again.  
    \item On the vertical walls (where $\frac{k-1}{2^{n-1}}<x_N<\frac{k}{2^{n-1}}$ for some $k$ even), by the boundary condition \eqref{conditionboundary}, we have $\displaystyle\frac{\partial\varphi_n}{\partial x_i}=\fint_{\frac{k-1}{2^{n-1}}}^{\frac{k}{2^{n-1}}}\frac{\partial}{\partial x_i}\varphi(x,z)dz= 0$ for $i \neq N$.  
\end{itemize}
Therefore, all boundary integrals vanish, leaving only the term
\[
\sum_{\substack{k=1 \\ k \text{ even}}}^{2^{n-1}} \int_{(S_n)_k} \frac{\partial u_n}{\partial x_i} \, \frac{\partial \varphi_n}{\partial x_i} \, dx
= - \sum_{\substack{k=1 \\ k \text{ even}}}^{2^{n-1}} \int_{(S_n)_k} u_n \, \frac{\partial^2 \varphi_n}{\partial x_i^2} \, dx
= - \int_{A_n} u_n \, \frac{\partial^2 \varphi_n}{\partial x_i^2} \, dx,
\]
and we can rewrite \eqref{localstrip} as follows:
\begin{align}
    &-\int_0^1\int_{[0,1]^{N-1}} \varphi(\hat{x},x_N)f_n(\hat{x},x_N)\, d\hat{x}\chi_{A_n}(x_N)\, dx_N-\int_0^1\int_{[0,1]^{N-1}}  \phi(\hat{x},x_N)f_n(\hat{x},x_N)\, d\hat{x}\chi_{B_n}(x_N)\, dx_N\notag\\
    &= -\sum_{i=1}^{N-1}\int_{[0,1]^N}\chi_{A_n}(x_N)u_n(x)\frac{\partial^2 \varphi_n}{\partial x_i^2}(x)dx +\iint_{[0,1]^{2N}} J(x-y)\chi_{B_n}(y_N)\chi_{A_n}(x_N)(v_n(y)-u_n(x))\phi(y)dydx\label{stripJnl}\\
    &\quad -\iint_{[0,1]^N\times [0,1]^N} J(x-y)\chi_{B_n}(y_N)\chi_{A_n}(x_N)(v_n(y)-u_n(x))\varphi_n(x)dydx\label{nlstrip}\\
    &\quad -\iint_{[0,1]^N\times [0,1]^N} G(x-y)\chi_{B_n}(y_N)\chi_{B_n}(x_N)(v_n(y)-v_n(x))\phi(x)dydx.\label{stripGnl}
\end{align}
Since $B_n$ is path--connected through its boundary, we apply Lemma~\ref{uvbounded} and consider the corresponding bounded sequences
\begin{equation}\label{bound}
\|u_n\|_{L^2(A_n)} \leq C, \qquad \|v_n\|_{L^2(B_n)}\leq C,
\end{equation} 
We have that, up to a subsequence, there exists $u,v\in L^2([0,1]^N)$ such that
\begin{equation}\label{stripconvergence}
u_n\rightharpoonup u\in L^2 ([0,1]^N), \qquad \text{and} \qquad v_n\rightharpoonup v \in L^2 ([0,1]^N)
\end{equation} 
then, we can deal with nonlocal term in \eqref{stripJnl} (and analogously for \eqref{nlstrip} and \eqref{stripGnl}) to pass to the limit when $n\to \infty$, since
\begin{align*}
&\iint_{[0,1]^{2N}} J(x-y)\chi_{B_n}(y_N)\chi_{A_n}(x_N)v_n(y)\varphi_n(x)dydx-\iint_{[0,1]^{2N}} J(x-y)\chi_{B_n}(y_N)\chi_{A_n}(x_N)u_n(x)\varphi_n(x)dydx\\
&=\int_{[0,1]^N}\bigg(\int_{[0,1]^N} J(x-y)\chi_{B_n}(y_N)v_n(y)dy\bigg)\varphi_n(x)\chi_{A_n}(x_N)dy\\
&\qquad \qquad \qquad \qquad \qquad \qquad \qquad \qquad \qquad -\int_{[0,1]^N}\bigg(\int_{[0,1]^N} J(x-y)\chi_{A_n}(x_N)u_n(x)\varphi_n(x)dx\bigg)\chi_{B_n}(y_N)dy.
\end{align*}
The second term on the right side can be handled using the continuity of $J$, the weak convergence 
$\chi_{A_n}u_n \rightharpoonup u$ in $L^2([0,1]^N)$, and the uniform convergence 
$\varphi_n \to \varphi$ in $C([0,1]^N)$.  
As a consequence, we obtain the strong convergence
\[
\int_{[0,1]^N} J(x-y)\chi_{A_n}(x_N)u_n(x)\varphi_n(x)\,dx
\;\longrightarrow\;
\int_{[0,1]^N} J(x-y)u(x)\varphi(x)\,dx
\quad \text{in } C([0,1]^N), \ \text{as}\  n\to \infty.
\]
Using the weak convergence $\chi_{B_n}\xrightharpoonup{*}1-X$ in $L^\infty([0,1]^N)$, we obtain
\begin{equation*}\label{nlargue}
\int_{[0,1]^N}\bigg(\int_{[0,1]^N} J(x-y)\chi_{A_n}(x_N)u_n(x)\varphi_n(x)dx\bigg)\chi_{B_n}(y_N)dy\longrightarrow \int_{[0,1]^N}\int_{[0,1]^N} J(x-y)u(x)\varphi(x)dx\,(1-X)dx,
\end{equation*}
as $n\to \infty$. The limit of the first term on the right side follows with similar computations. 

Finally, taking into account relations \eqref{stripconvergence} and \eqref{specialfunctionstripcv} in the local term \eqref{stripJnl}, and using the arguments above to handle the non-local terms \eqref{stripJnl}, \eqref{nlstrip}, and \eqref{stripGnl}, together with the strong convergence $f_n \to f$ in $L^2([0,1]^N)$, with the convergence
\begin{align*}
 &\int_{0}^{1}\!\!\int_{[0,1]^{N-1}} 
 \varphi_n(\hat{x},x_N)f_n(\hat{x},x_N)\,d\hat{x}\,\chi_{A_n}(x_N)\,dx_N
 +\int_{0}^{1}\!\!\int_{[0,1]^{N-1}} 
 \varphi_n(\hat{x},x_N)f_n(\hat{x},x_N)\,d\hat{x}\,\chi_{B_n}(x_N)\,dx_N \\
& \longrightarrow 
\int_{[0,1]^N} Xf(x)\varphi(x)\,dx
+\int_{[0,1]^N} (1-X)f(x)\varphi(x)\,dx,
\qquad \text{as } n \to \infty,
\end{align*}
we obtain the distributional formulation of the limit problem that reads as, 
\begin{align*}
    &-\int_{[0,1]^N} X f(x) \varphi(x)\,dx
    -\int_{[0,1]^N} (1-X) f(x) \phi(x)\,dx = -\sum_{i=1}^{N-1} \int_{[0,1]^N} 
    u(x) \frac{\partial^2 \varphi}{\partial x_i^2}(x)\,dx \\
    &\quad +\int_{[0,1]^N}\!\!\bigg[\int_{[0,1]^N} 
    J(x-y)\bigl(Xv(y)-(1-X)u(x)\bigr)\,dy\bigg]\phi(y)\,dx\\
    &\qquad -\int_{[0,1]^N}\!\!\bigg[\int_{[0,1]^N} 
    J(x-y)\bigl(Xv(y)-(1-X)u(x)\bigr)\,dy\bigg]\varphi(x)\,dx \\
    &\quad -\int_{[0,1]^N}\!\!\bigg[\int_{[0,1]^N} 
    G(x-y)\bigl((1-X)v(y)-(1-X)v(x)\bigr)\,dy\bigg]\phi(x)\,dx.
\end{align*}

Now, let $(\varphi(x),0) = (\xi(\hat{x})\, \psi(x_N),0)$ with 
$\int_{[0,1]^N} \varphi \, dx = 0$. Then, for $i=1,\dots,N-1$,
\[
\frac{\partial^2 \varphi}{\partial x_i^2}(x) = \frac{\partial^2 \xi}{\partial x_i^2}(\hat{x}) \, \psi(x_N),
\]
since $\psi$ depends only on $x_N$.
Using these test functions in the distributional formulation, we have
\begin{align}
\int_0^1 \psi(x_N) 
&\int_{[0,1]^{N-1}} X(x_N) f(x) \, \xi(\hat{x}) \, d\hat{x} \, dx_N
= \sum_{i=1}^{N-1} \int_0^1 \psi(x_N) 
\int_{[0,1]^{N-1}} u(x) \frac{\partial^2 \xi}{\partial x_i^2}(\hat{x}) \, d\hat{x} \, dx_N \notag\\
&\quad + \int_0^1 \psi(x_N) \int_{[0,1]^{N-1}} 
\Bigg[ \int_{[0,1]^N} J(x-y) \big( X(x_N) v(y) - (1-X(y_N)) u(x) \big) dy \Bigg] 
\xi(\hat{x}) \, d\hat{x} \, dx_N
\end{align}
for all $\psi \in C([0,1])$. By the density of $C([0,1])$ in $L^1([0,1])$ and Fubini's theorem, we deduce that for almost every $x_N \in [0,1]$,
\begin{align}
\int_{[0,1]^{N-1}} X(x_N) f(x) \, \xi(\hat{x}) \, d\hat{x} 
&= \sum_{i=1}^{N-1} \int_{[0,1]^{N-1}} u(x) \frac{\partial^2 \xi}{\partial x_i^2}(\hat{x}) \, d\hat{x} \notag\\
&\quad + \int_{[0,1]^{N-1}} \Bigg[ \int_{[0,1]^N} J(x-y) \big( X(x_N) v(y) - (1-X(y_N)) u(x) \big) dy \Bigg] 
\xi(\hat{x}) \, d\hat{x},
\end{align}
for all $\xi \in C^\infty([0,1]^{N-1})$. Notice that
\begin{align}
\underbrace{X(x_N) f(x) - \int_{[0,1]^N} J(x-y) \big( X(x_N) v(y) - (1-X(y_N)) u(x) \big) dy}_{\in L^2([0,1]^{N-1})}
\end{align}
serves as the source term in the weak Laplacian in the $\hat{x}$-directions. From standard elliptic regularity in $[0,1]^{N-1}$, we conclude that for almost every $x_N$,
\[
u(\cdot, x_N) \in H^2([0,1]^{N-1}).
\]
Therefore, we obtain the Sobolev inclusion
\[
u \in L^2\big([0,1]_{x_N}; H^1([0,1]^{N-1})\big).
\]
This shows that $(u,0)$ is a weak solution of \eqref{hstripslocal} in the $\hat{x}$-directions. 

On the other hand, for every $(0,\phi) \in H^1(\Omega)\times L^2(\Omega)$, with $\int_\Omega \phi(x)dx=0$, we have
\begin{align}
    -\int_{[0,1]^N} (1-X(x_N))&f(x)\phi(x)dx=\int_{[0,1]^N}\bigg[\int_{[0,1]^N} J(x-y)[X(x_N)v(y)-(1-X(y_N))u(x)]dy\bigg]\phi(y)dx,\notag\\
    &-\int_{[0,1]^N}\bigg[\int_{[0,1]^N} G(x-y)[(1-X(x_N))v(y)-(1-X(y_N))v(x)]dy\bigg]\phi(x)dx.
\end{align}
With the same arguments mentioned above, we obtain the weak formulation of the system \eqref{hstripslocal}-\eqref{hstripsnl}, whose variational formula is associated with \eqref{lmlimit}. 

As the same argument leading to uniqueness of the solution used in the proof of Theorem \ref{localinballs} and Theorem \ref{thrmnlinballs} can be also used here, we can also verify that $(u,v)\in L^2((0,1);H^1(0,1)^{N-1})\times L^2([0,1]^N)$ is the unique solution to the system \eqref{hstripslocal}--\eqref{hstripsnl}.
\end{proof}

We can extend these computations in some contexts. Here we mention the case of a $N$ dimensional equidistributed cylinder per cell.

\begin{remark} {\rm Now, we describe a generalization of the previous computation with the following setting: 
Let $[0,1]^N$ be the unit cube in $\mathbb{R}^N$ and let $l<N$. We divide the cross-section into $n^{N-l}$ cells and place one cylinder per cell, for $n$ fixed.
Let $B(x_i,r_n) \subset \mathbb{R}^{N-l}$ in the $i$-th cell of the ${N-l}$-dimensional grid and $r_n < 1/n$. Extend each ball along the remaining last coordinate to form a cylinder. Then we define the periodic cylindrical domain
\[
A_n = \bigcup_{i=1}^{n^{N-l}}A_i:=  \bigcup_{i=1}^{n^{N-l}}
\Bigl( B(x_i,r_n) \times [0,1]^l  \Bigr) \cap [0,1]^{N} \,.
\]
The complementary domain is
$$B_n = [0,1]^N \setminus A_n.$$
For example, for $N=3$ and $l=2$ (cylinders along the $x_3$ axis):
\[
A_n = \bigcup_{i=1}^{n^2} B(x_i,r_n) \times [0,1] \subset [0,1]^3, 
\quad B_n = [0,1]^3 \setminus A_n.
\]
This setting corresponds to a thick junction of type 3:2:1, in the notation of \cite{taras}.

The homogenization limit follows once we make a small but important tweak in the construction of the special class introduced in \eqref{specialtestfunctionstrip}, given by $\{\varphi_n:[0,1]^N\to \mathbb{R}\}$. Let us consider $\varphi \in C([0,1]^l;C^2([0,1]^{N-l})$ with $$\int_{[0,1]^N} \varphi(x) dx =0$$ which satisfies 
\begin{equation}\frac{\partial \varphi}{\partial \eta}=0 \in \partial A_n, \qquad\frac{\partial \varphi}{\partial x_i}(x)\bigg|_{x_i=0}=0=\frac{\partial \varphi}{\partial x_i}(x)\bigg|_{x_i=1}, \qquad \text{for all} \quad  x_l,\ldots,x_N\in [0,1], \quad i=1,\ldots,l.\end{equation} 
We define
\begin{align}
\varphi_n(x)=\begin{cases}        \varphi(x), \ x\in B_n\\ 
       \displaystyle \fint_{A_i}\varphi(z_1,\ldots,z_l,z_{l+1},\ldots, z_N)dz_1\ldots \ d z_l\ , \ x\in A_i.
    \end{cases}
\end{align}

With this test function we can reply the same computations given in the proof of Theorem \ref{thrmstrip}.
We leve the details to the reader. }
\end{remark}

\vspace{0.5cm}

\noindent\textbf{Acknowledgments:} 
Partially supported by Grants CNPq \#303561/2024-6 and FAPESP \#20/14075-6 (Brazil) and 
UBACyT 20020160100155BA (Argentina) and CONICET PIP GI No 11220150100036CO (Argentina). 

The first author MCP gratefully acknowledges the medical team led by Prof. Dr. G. Lepski (HC/FMUSP) for their outstanding care, 
which was instrumental in saving his life.

\end{document}